\documentclass[12pt]{article}

\usepackage{amsmath,amsthm,amssymb}
\allowdisplaybreaks
\usepackage{mathrsfs}
\usepackage{cite}

\usepackage[pdftex,bookmarksopen=true,bookmarks=true,unicode,setpagesize]{hyperref}
\hypersetup{colorlinks=true,linkcolor=black,citecolor=black}

\makeatletter\@addtoreset{equation}{section}\makeatother

\newtheorem{theorem}{Theorem}[section]
\newtheorem{corollary}[theorem]{Corollary}
\newtheorem{lemma}[theorem]{Lemma}
\newtheorem{proposition}[theorem]{Proposition}

\theoremstyle{remark}

\theoremstyle{remark}

\theoremstyle{remark}
\newtheorem{remark}[theorem]{Remark}

\newcommand{\cd}{\circledast}
\newcommand{\la}{\langle}
\newcommand{\ra}{\rangle}
\newcommand{\h}{\mathcal{H}}

\newcommand{\di}{\partial}
\newcommand{\R}{{\mathbb R}}
\newcommand{\Sym}{\operatorname{Sym}}
\newcommand{\id}{\operatorname{id}}

\newcommand{\SSym}{\mathbf{Sym}}

\begin{document}

\vspace{-20mm}
\begin{center}{\large \bf
An extended anyon Fock space and noncommutative Meixner-type orthogonal polynomials in infinite dimensions}\\[4mm]
{\it Dedicated to Professor Anatoly Moiseevich Vershik\\
on the occasion of his 80th birthday}
\end{center}

{\large Marek Bo\.zejko}\\
Instytut Matematyczny, Uniwersytet Wroc{\l}awski, Pl.\ Grunwaldzki 2/4, 50-384 Wroc{\l}aw, Poland; e-mail: \texttt{bozejko@math.uni.wroc.pl}\vspace{2mm}

{\large Eugene Lytvynov}\\ Department of Mathematics,
Swansea University, Singleton Park, Swansea SA2 8PP, U.K.;
e-mail: \texttt{e.lytvynov@swansea.ac.uk}\vspace{2mm}

{\large Irina Rodionova}\\ Department of Mathematics,
Swansea University, Singleton Park, Swansea SA2 8PP, U.K.;
e-mail: \texttt{i.rodionova@swansea.ac.uk}\vspace{2mm}


{\small

\begin{center}
{\bf Abstract}
\end{center}
\noindent Let $\nu$ be a finite measure on $\mathbb R$ whose Laplace transform is analytic in a neighborhood of zero. An anyon L\'evy white noise on $(\mathbb R^d,dx)$ is a certain family of noncommuting operators $\langle\omega,\varphi\rangle$ in the anyon Fock space over $L^2(\mathbb R^d\times\mathbb R,dx\otimes\nu)$. Here $\varphi=\varphi(x)$ runs over a space of test functions on $\mathbb R^d$, while $\omega=\omega(x)$ is interpreted as an operator-valued distribution on $\mathbb R^d$. Let $L^2(\tau)$ be the noncommutative $L^2$-space generated by the algebra of polynomials in variables $\langle \omega,\varphi\rangle$, where $\tau$ is  the vacuum expectation state. We construct noncommutative orthogonal polynomials in $L^2(\tau)$ of the form $\langle P_n(\omega),f^{(n)}\rangle$, where $f^{(n)}$ is a test function on $(\mathbb R^d)^n$. Using these orthogonal polynomials, we derive a unitary isomorphism $U$ between $L^2(\tau)$ and an extended anyon Fock space over $L^2(\mathbb R^d,dx)$, denoted by $\mathbf F(L^2(\mathbb R^d,dx))$. The usual anyon
 Fock space over $L^2(\mathbb R^d,dx)$,  denoted by $\mathcal F(L^2(\mathbb R^d,dx))$, is a subspace of $\mathbf F(L^2(\mathbb R^d,dx))$. Furthermore, we have the equality $\mathbf F(L^2(\mathbb R^d,dx))=\mathcal F(L^2(\mathbb R^d,dx))$ if and only if the measure $\nu$ is concentrated at one point, i.e., in the Gaussian/Poisson case. Using the unitary isomorphism $U$, we realize the operators $\langle \omega,\varphi\rangle$ as a Jacobi (i.e., tridiagonal) field in $\mathbf F(L^2(\mathbb R^d,dx))$. We derive a Meixner-type class of anyon L\'evy white noise for which the respective Jacobi field in $\mathbf F(L^2(\mathbb R^d,dx))$ has a relatively simple structure. Each anyon L\'evy white noise  of the Meixner type is characterized by  two parameters: $\lambda\in\mathbb R$ and $\eta\ge0$. Furthermore, we get the representation
$\omega(x)=\partial_x^\dag+\lambda \partial_x^\dag\partial_x +\eta\partial_x^\dag\partial_x\partial_x+\partial_x$.
Here $\partial_x$ and $\partial_x^\dag$ are   annihilation and creation operators at point $x$. } \vspace{2mm}

\newpage

\section{Meixner polynomials in infinite dimensions}

\subsection{Meixner class of orthogonal polynomials}

In  1934, Meixner \cite{Meixner} studied the following  problem. Consider complex-valued functions $u(z)$ and $\Phi(z)$ which can be expanded into a  power
series of $z\in\mathbb C$ in a neighborhood of zero and suppose that $u(0)=1$, $\Phi(0)=0$, and $\Phi'(0)=1$. Then the function
\begin{equation}\label{cyd6i}
 G(x,z)=\exp\big[x\Phi(z)\big]u(z)=\sum_{n=0}^\infty \frac{P_n(x)}{n!}\,z^n\end{equation}
generates a system of monic polynomials $P_n(x)$. Find all such polynomials which are orthogonal with respect to a probability measure $\mu$ on $\R$. Such polynomials are sometimes called  orthogonal polynomials with generating function of exponential type.  

Meixner \cite{Meixner} proved that a system of polynomials $P_n(x)$ belongs to this class if and only if it satisfies the recurrence relation
\begin{equation}\label{uyfu87t}
xP_n(x)=P_{n+1}(x)+(l+n\lambda)P_{n}(x)+n(k+\eta(n-1))P_{n-1}(x),\quad n\in\mathbb N_0,\end{equation}
where $l\in\R$, $k>0$, $\lambda\in\R$, $\eta\ge0$. 
For each choice of the parameters, the corresponding measure of orthogonality, $\mu$, is infinitely divisible. If $l=0$,  $\mu$ becomes centered, whereas $l\ne0$ corresponds to the shift of $\mu$ by $l$. For $l=0$ and $k\ne1$, the  measure $\mu$ is the $k$-th convolution power of the corresponding measure $\mu$ for $k=1$. 

One distinguishes five classes of polynomials satisfying \eqref{uyfu87t} (see \cite{Meixner, Chihara}):

(i) For $\lambda=\eta=0$, $\mu$ is a Gaussian measure,  $(P_n)_{n=0}^\infty$ is a system of Hermite polynomials.

(ii) For $\lambda\ne0$ and $\eta=0$, $\mu$ is similar to a Poisson distribution ($\mu$ being a real Poisson distribution  when $\lambda=1$ and  $l=1$),  $(P_n)_{n=0}^\infty$ is a system of Charlier polynomials.

(iii) For $|\lambda|=2$ and $\eta\ne0$, $\mu$ is a gamma distribution, 
$(P_n)_{n=0}^\infty$ is a system of Laguerre polynomials.

(iv) For $|\lambda|<2$ and $\eta\ne0$, $\mu$ is a Pascal (negative binomial) distribution, $(P_n)_{n=0}^\infty$ is a system of Meixner polynomials of the first kind. 

(v) For $|\lambda|>2$ and $\eta\ne0$, $\mu$ is a Meixner distribution, $(P_n)_{n=0}^\infty$ is a system of Meixner polynomials of the second kind, or Meixner--Polaczek polynomials.  

Note that, in each case, for $z$ from a neighborhood of zero in $\mathbb C$,
\begin{equation}\label{ftse5sw5}
 G(x,z)=\exp\big[x\Phi(z)-\mathcal C(\Phi(z))\big], \end{equation}
where $\mathcal C(z):=\log\left(\int_{\R}e^{xz}\,\mu(dx)\right)$ is the cumulant transform of $\mu$. We refer to \cite{Meixner, Chihara} for explicit formulas of $\Phi(z)$ and $\mathcal C(z)$. If one introduces complex parameters $\alpha,\beta\in\mathbb C$ such that   
$\alpha+\beta=-\lambda$ and $\alpha\beta=\eta$, using Taylor's expansion, one can write down explicit formulas for  $\Phi(z)$ and $\mathcal C(z)$ in a unique form for all parameters $\alpha$ and $\beta$, see \cite{Rodionova}. 

The two observations below will be crucial for our considerations. 
First, setting $l=0$ and $k=1$, we can rewrite formula \eqref{uyfu87t} as follows
\begin{equation}\label{txsjusxt} x=\di^\dag+\lambda \di^\dag\di +\di+\eta\di^\dag\di\di.\end{equation}
Here (with an abuse of notation) $x$ denotes the operator of multiplication by the variable $x$ in $L^2(\R,\mu)$, $\di^\dag$ is a creation (raising) operator: $\di^\dag P_n(x)=P_{n+1}(x)$,  and $\di$ is an annihilation (lowering) operator:  $\di P_n(x)=nP_{n-1}(x)$.

Second, Kolmogorov's representation of the Fourier transform of the infinitely divisible measure $\mu$ (with $l=0$) has the form  \cite{Schoutens,Schoutens_Teugels} 
$$\int_\R e^{iux}\,\mu(dx)=\exp\left[k\int_\R
(e^{ius}-1-ius)s^{-2}\,\nu(ds)\right],\quad u\in\R, $$
see also \cite{Grigelionis}. 
Here, for $\eta=0$ (Gaussian and Poisson cases), $\nu=\delta_\lambda$, the Dirac measure with mass at $\lambda$, whereas for $\eta\ne0$ (cases (iii)--(v)) $\nu$ is the probability measure on $\R$, whose system of monic orthogonal polynomials, $(p_n)_{n=0}^\infty$, satisfies the recurrence formula
\begin{equation}\label{noh9y}
sp_n(s)=p_{n+1}(s)+\lambda(n+1)s+\eta n(n+1)p_{n-1}(s).
\end{equation} 
In particular, $(p_n)_{n=0}^\infty$ is again a system of orthogonal polynomials from the Meixner class.

\subsection{An infinite dimensional extension}\label{fytry}

It appears that the Meixner class of orthogonal polynomials is fundamental for infinite dimensional analysis, in particular, for the theory of L\'evy white noise, see e.g.\ \cite{afs,Ly2,L2,Schoutens,TsVY}  and the references therein. Let 
$X:=\R^d$ and let 
$$\mathscr D(X)\subset L^2(X,dx)\subset\mathscr D'(X)$$ be a standard triple of spaces in which $\mathscr D(X)$ is the nuclear space of smooth, compactly supported functions on $X$ and
$\mathscr D'(X)$  is the dual space of $\mathscr D(X)$ with respect to zero space $L^2(X,dx)$.
For  $\omega\in\mathscr D'(X)$ and $\varphi\in\mathscr D(X)$, we denote by $\la \omega,\varphi\ra$ the dual pairing between   $\omega$ and $\varphi$. Let $\mu$ be 
a probability measure  on $\mathscr D'(X)$, and assume that  $\mu$ is a generalized stochastic process with independent values, in the sense of \cite{GV}, or using another terminology, a L\'evy white noise measure \cite{DOP}. 
We will assume that $\mu$ is centered and its Fourier transform has Kolmogorov's representation 
\begin{equation}\label{jig8yugtf8}\int_{\mathscr D'(X)}e^{i\la\omega,\varphi\ra}\mu(d\omega)=\exp\bigg[\int_{X}\int_\R \big(e^{is\varphi(x)}-1-is\varphi(x)\big)s^{-2}\nu(ds)\,dx\bigg],\quad \varphi\in \mathscr D(X),\end{equation}
where 
$\nu$ is a probability measure on $\R$ which satisfies:
\begin{equation}\label{ft7er7i57}\int_{\R}e^{\varepsilon|s|}\,\nu(ds)<\infty\quad \text{for some $\varepsilon>0$.}\end{equation} 
 Note that the measure $s^{-2}\nu(ds)$ on $\R\setminus\{0\}$ is called the L\'evy measure of $\mu$, while $\nu(\{0\})$  describes the Gaussian part of  $\mu$ (for $s=0$, the function under the integral sign in \eqref{jig8yugtf8} is equal to $-(1/2)\varphi^2(x)$).

In the case $d=1$, for each $t\ge0$, one can define by approximation in $L^2(\mathscr D'(X),\mu)$ a random variable $L_t(\omega)=\la\omega,\chi_{[0,t]}\ra$. Here $\chi_{[0,t]}$ denotes the indicator function of $[0,t]$. Then $(L_t)_{t\ge0}$ is a (version of a) L\'evy process with Kolmogorov measure $\nu$:
$$ \int_{\mathscr D'(X)}e^{iuL_t(\omega)}\,\mu(d\omega)=
\exp\left[t\int_\R (e^{ius}-1-ius)s^{-2}\nu(ds)\right].$$
Thus, the measure $\mu$ is indeed a L\'evy white noise.

Denote by $\mathscr {CP}$ the set of all continuous polynomials on $\mathscr D'(X)$, i.e., functions on $\mathscr D'(X)$ of the form
\begin{equation}\label{ytfdytdey7} f^{(0)}+\sum_{i=1}^n\la \omega^{\otimes i},f^{(i)}\ra,\quad \omega\in\mathscr D'(X),\ f^{(0)}\in\R,\ f^{(i)}\in \mathscr D(X)^{\otimes i},\ i=1,\dots,n,\ n\in\mathbb N.\end{equation}
If $f^{(n)}\ne0$, one says that the polynomial in \eqref{ytfdytdey7} has order $n.$
The set 
 $\mathscr{CP}$ is  dense in  $L^2(\mathscr D'(X),\mu)$.
So using the approach proposed by Skorohod \cite{Sko}, we may  
orthogonalize these polynomials. More precisely, we denote by $\mathscr {CP}_n$ the linear space of all continuous polynomials on $\mathscr D'(X)$  of order $\le n$.
Let $\mathscr{MP}_n$ denote the closure of $\mathscr {CP}_n$ in $L^2(\mathscr D'(X),\mu)$ (the set of measurable polynomials of order $\le n$). Let $\mathscr{OP}_n:=\mathscr{MP}_n\ominus \mathscr{MP}_{n-1}$, the set of orthogonalized polynomials on $\mathscr D'(X)$ of order $n$. 
We clearly have
\begin{equation}\label{gyut8t}
L^2(\mathscr D'(X),\mu)=\bigoplus_{n=0}^\infty \mathscr{OP}_n. \end{equation}
 
 \begin{remark}
 An alternative orthogonal decomposition of the $L^2$-space of a L\'evy process was derived by  Vershik and Tsilevich in \cite{VTs}.
 \end{remark}

 For each $f^{(n)}\in \mathscr D(X)^{\otimes n}$, we denote by
$\la P_n(\omega), f^{(n)}\ra$ the orthogonal projection of the continuous monomial $\la \omega^{\otimes n}, f^{(n)}\ra$ onto $\mathscr{OP}_n$.   We denote by  $\mathscr {OCP}$ the linear space of orthogonalized continuous polynomials, i.e., the space of finite sums of functions of the form
$\la P_n(\omega),f^{(n)}\ra$ and constants. It should be stressed that the function $\la P_n(\omega), f^{(n)}\ra$ does not necessarily belong to $\mathscr{CP}$. 

\begin{theorem}\label{ur7o67r6}
Let $\mu$ be a probability measure on $\mathscr D'(X)$ which has Fourier transform  \eqref{jig8yugtf8} with $\nu$ being a probability measure on $\mathbb R$ satisfying \eqref{ft7er7i57}. 
Then we have 
$$\mathscr{CP}=\mathscr {OCP}$$
 if and only if there exist  $\lambda\in\R$ and $\eta\ge0$ such that, if $\eta=0$ then $\nu=\delta_\lambda$, and if $\eta>0$ then 
 the system of monic polynomials $(p_n)_{n=0}^\infty$ which are orthogonal with respect to the measure $\nu$ satisfies the  recurrence formula
\eqref{noh9y} with $\lambda$ and $\eta$.
 \end{theorem}

This theorem can be derived from the main result of \cite{BML}. It will also be a corollary of Theorem~\ref{urr8r} below. 

We define the generating function of the orthogonal polynomials by
$$G_\mu(\omega,\varphi):=\sum_{n=0}^\infty \frac1{n!}\la P_n(\omega),\varphi^{\otimes n}\ra,$$
and the cumulant transform of the measure $\mu$ by
$$\mathcal C_\mu(\varphi):=\log\left(\int_{\mathscr D'(X)}e^{\la\omega,\varphi\ra}\mu(d\omega)\right).$$  
The following theorem shows, in particular,  that formula  \eqref{ftse5sw5} admits an extension to infinite dimensions, see \cite{Ly2} for a proof.

\begin{theorem}\label{igt8t} Fix any $\lambda\in\R$ and $\eta\ge0$. 
Let $\mu$ be the  probability measure  on $\mathscr D'(X)$ 
which has Fourier transform  \eqref{jig8yugtf8} with $\nu$ being the probability measure on $\R$ corresponding to the parameters $\lambda$ and $\eta$ as in 
 Theorem~\ref{ur7o67r6}.
 Let $\mathcal C(\cdot)$ and $\Phi(\cdot)$ be the functions as in \eqref{ftse5sw5} for parameters $l=0$, $k=1$ and $\lambda$ and $\eta$ as above. Then
\begin{align*}
\mathcal C_\mu(\varphi)&=\int_{X}\mathcal C(\varphi(x))\,dx,\\
G_\mu(\omega,\varphi)&=\exp\left[
\la \omega,\Phi(\varphi)\ra-\int_{X}\mathcal C(\Phi(\varphi(x)))\,dx\right], 
\end{align*}
the formulas hold for $\varphi$ from (at least) a neighborhood of zero in $\mathscr D(X)$.
\end{theorem}

In the case $\lambda=0$, $\eta=0$, $\mu$ is a Gaussian white noise measure. We refer to e.g.\ \cite{BK,DOP,HKPS} for Gaussian white noise analysis. 

In the case  $\lambda\ne0$ and  $\eta=0$,   $\mu$ is a Poisson random measure (or point process), see e.g.\ \cite{Kal}. 
We refer to \cite{VGG} for  a discussion of representations of the group of diffeomorphisms in the Poisson space, to 
\cite{IK} for Poisson white noise analysis, and to \cite{AKR} for Poisson analysis on the configuration space. 

For $\eta\ne 0$, the most important case of $\mu$ is when $\lambda=2$  and $\eta=1$. Then $\mu$ is the centered gamma measure. The gamma measure is concentrated on discrete Radon measure on $X$, $\sum_i s_i\delta_{x_i}$, such that the configuration of atoms, $\{x_i\}$, is a dense subset of $X$. A very important property of the gamma measure is that it is quasi-invariant with respect to a natural group of transformations  of the weights, $s_i$, see  \cite{TsVY} and the references therein. Furthermore, as shown in \cite{TsVY}, the gamma measure is the unique law of a measure-valued L\'evy process which has an equivalent $\sigma$-finite measure which is projective invariant with respect to the action of the group acting on the weights, $s_i$. This $\sigma$-finite measure is called in \cite{TsVY} the infinite dimensional Lebesgue measure, see also \cite{Vershik}. We also note that, in papers \cite{TsVY,VGG1,VGG2,VGG3}, the gamma measure was used in the representation theory of the group $SL(2,F)$, where $F$ is an algebra of functions on a manifold. White noise analysis related to the gamma measure was initiated  in \cite{KSSU}, and further developed in \cite{KL}. Gibbs perturbations of the gamma measure were constructed in 
\cite{HKPR}. A Laplace operator associated with the gamma measure was constructed and studied in \cite{HKLV}. 
Finally,  infinite dimensional analysis related to the case of a general    $\eta\ne0$ was studied in \cite{Ly1,Ly2}.

It is well known that, in the Gaussian and  Poisson cases ($\eta=0$), the decomposition of $L^2(\mathscr D'(X),\mu)$ in orthogonal polynomials yields the Wiener--It\^o--Segal isomorphism between $L^2(\mathscr D'(X),\mu)$ and the symmetric Fock space over $L^2(X,dx)$. (An alternative derivation of this result is achieved by using multiple stochastic integrals, see e.g.\ \cite{Surgailis} for the Poisson case.) This result admits the following extension, see \cite{KSSU,KL,Ly2}.   

\begin{theorem} \label{uyt8t8} 
Let $\lambda\in\R$ and $\eta\ge0$, and let $\mu$ be the corresponding probability measure on $\mathscr D'(X)$ as in  Theorem~\ref{igt8t}.

{\rm (i)}  For each $n\in\mathbb N$, there exists a measure $m_\nu^{(n)}$ on $X^n$ which satisfies
\begin{equation}\label{gfuyfrluf}\int_{\mathscr D'(X)}\la P_n(\omega), f^{(n)}\ra ^2\,\mu(d\omega)=\int_{X^n}(\operatorname{Sym}_n f^{(n)})^2\, dm_\nu^{(n)},\quad f^{(n)}\in\mathscr D(X)^{\otimes n}.\end{equation}
Here $\operatorname{Sym}_n f^{(n)}$ denotes the usual symmetrization of a function $f^{(n)}$. For $\eta=0$, $m_\nu^{(n)}=\frac1{n!}\,dx_1\dotsm dx_n$, for $\eta\ne0$ see subsec.~\ref{utf7r5svgy} below for the explicit construction of $m_\nu^{(n)}$.

{\rm (ii)} We define a Hilbert space
\begin{equation}\label{urtur}\mathbf F_{\mathrm{sym}}(L^2(X,dx),\nu):=\R\oplus\bigoplus_{n=1}^\infty L^2_{\mathrm{sym}}(X^n,m_\nu^{(n)}),\end{equation}
where $L^2_{\mathrm{sym}}(X^n,m_\nu^{(n)})$ is the subspace of $L^2(X^n,m_\nu^{(n)})$ consisting of all symmetric functions from this space. For $\eta=0$, $\mathbf F_{\mathrm{sym}}(L^2(X,dx),\nu)$ is the symmetric Fock space over 
$L^2(X,dx)$. For $\eta\ne0$, $\mathbf F_{\mathrm{sym}}(L^2(X,dx),\nu)$ contains the symmetric Fock space as a proper subspace. We then call $\mathbf F_{\mathrm{sym}}(L^2(X,dx),\nu)$ an extended symmetric Fock space. The mapping
\begin{equation}\label{huhgiuyfg}
f^{(0)}+\sum_{i=i}^n\la P_i(\omega),f^{(i)}\ra\mapsto(f^{(0)},\,\Sym_1f^{(1)},\dots,\,\Sym_nf^{(n)},0,0\dots)\in\mathbf F_{\mathrm{sym}}(L^2(X,dx),\nu)\end{equation}
extends by continuity to a unitary operator
$ U:L^2(\mathscr D'(X),\mu)\mapsto \mathbf F_{\mathrm{sym}}(L^2(X,dx),\nu)$.

{\rm (iii)} For each $\varphi\in\mathscr D(X)$, we keep the notation
$\la \omega,\varphi\ra$ for the image of the operator of multiplication by the monomial $\la \omega,\varphi\ra$ in $L^2(\mathscr D'(X),\mu)$ under the unitary operator $U$.
Then, analogously to \eqref{txsjusxt}, we have the following representation of the operator $\la \omega,\varphi\ra$ realized in the (extended) symmetric Fock space $\mathbf F_{\mathrm{sym}}(L^2(X,dx),\nu)$:
\begin{equation}\label{yufru7rf}\la \omega,\varphi\ra=\int_{X}dx\,\varphi(x)
(\di_x^\dag+\lambda \di_x^\dag\di_x +\di_x+\eta\di_x^\dag\di_x\di_x).\end{equation}
Here $\di_x$ is the annihilation operator at point $x$:
\begin{equation}\label{kgiygti9}(\di_x f^{(n)})(x_1,\dots,x_{n-1}):=n f^{(n)}(x,x_1,\dots,x_{n-1}),\end{equation}
and $\di_x^\dag$ is the creation operator at point $x$, satisfying
\begin{equation}\label{ufgutfr8ub}\int_{X}dx\,\varphi(x)\di_x^\dag\, f^{(n)}:=\operatorname{Sym}_{n+1}(\varphi\otimes f^{(n)}),\end{equation}
see \cite{Ly2} for further details.
\end{theorem}

Note that, in view of formula \eqref{yufru7rf}, we may heuristically write
\begin{equation}\label{ytr765r}
\omega(x)=\di_x^\dag+\lambda \di_x^\dag\di_x +\di_x+\eta\di_x^\dag\di_x\di_x.\end{equation}

As follows from Theorem \ref{uyt8t8}, (iii), the operators 
$\la \omega,\varphi\ra$ realized in $\mathbf F_{\mathrm{sym}}(L^2(X,dx),\nu)$ form a Jacobi field, i.e., 
 they have a tridiagonal structure; compare with e.g.\ \cite{B,BML,Br1,Br2,Ly1}.

\subsection{A noncommutative extension for anyons --- an introduction}

The above discussed results 
 have  noncommutative analogs in the framework of free probability \cite{BL1,BL2}, see also \cite{a2,a5,bb,bd,Biane} and the references therein. See also 
 \cite{BBLS,BH} for further connections between the classical distributions from the Meixner class and free probability.  

However, in this paper, we will be interested in a noncommutative extension of Meixner polynomials for  a so-called anyon statistics \cite{LM,GM,GS}, see also \cite{Bozejko}. The latter statistics, indexed by a complex number $q$ of modulus one, 
forms a continuous bridge between the boson statistics ($q=1$) and the fermi statistics ($q=-1$). One of the main aims of the present paper is to show that, in the anyon  setting, one  naturally arrives at noncommutative Meixner-type polynomials which have a representation like in \eqref{yufru7rf}.

In fact, one could think that it was hopeless to expect
a counterpart of formula \eqref{yufru7rf} in the fermion setting. Indeed,  if the operators $\di_x$ and $\di_y$ anticommute,
i.e.,  $\di_x\di_y=-\di_y\di_x$,
then  $\di_x\di_x=0$, so that the term $\eta\di_x^\dag\di_x\di_x$ must be equal to zero. However, we do show that, even in the fermion setting, the integral $\int_{X}dx\,\varphi(x)\di_x^\dag\di_x\di_x$ leads to a well-defined, nontrivial operator in an extended antisymmetric  Fock space $\mathbf F_{\mathrm{as}}(L^2(X,dx),\nu)$. The latter space contains the usual
antisymmetric (fermion)  Fock space  $\mathcal F_{\mathrm{as}}(L^2(X,dx))$ as a  subspace. On the space $\mathcal F_{\mathrm{as}}(L^2(X,dx))$, the operators $\di_x$ and $\di_y$ indeed anticommute. However, this anticommutaion fails on the whole  space   $\mathbf F_{\mathrm{as}}(L^2(X,dx),\nu)$. As a result, the extended antisymmetric Fock space leads to a proper renormalization (rather a nontrivial extension) of the operators $\di_x$ and $\di^\dag_x$.

Our discussion of this noncommutative extension is organized as follows. In Section~\ref{uiy976}, following \cite{GM,BLW,LM}, we  briefly recall 
the construction of the
anyon Fock space,  standard operators on them, and the anyon commutation relations. We  also recall the construction of a L\'evy white noise for anyon statistics as a family of noncommutative self-adjoint operators $\la\omega,\varphi\ra$ in the anyon Fock space over $L^2(X\times\R,dx\,\nu(ds))$, see \cite{BLW} for details. Note that, in this section, we do not explain why the `increments' of this process can  be understood as being `anyon independent.' For this, we refer the reader to \cite{BLW}. We  only note that in the commutative, boson setting ($q=1$), we  indeed recover a classical L\'evy white noise, being realized as a family of commuting self-adjoint operators in the symmetric Fock space over $L^2(X\times\R,dx\,\nu(ds))$.

In Section~\ref{yre6tr5}, we   formulate the main results of the paper. In particular, starting with a space $\mathscr{CP}$ of noncommutative continuous  polynomials of anyon white noise, we  construct 
a space $\mathscr{OCP}$ of orthogonalized 
continuous polynomials. By analogy with \eqref{gfuyfrluf}, for each $n\in\mathbb N$, we  construct a measure $m_\nu^{(n)}$ on $X^n$ and find the corresponding symmetrization operator $\operatorname{Sym}_n$. This symmetry extends the anyon symmetry (in particular, the fermion symmetry) in a  non-trivial way. By analogy with   \eqref{urtur}, we  define an extended anyon Fock space, and then by analogy with \eqref{huhgiuyfg}, we  construct a unitary operator  $U$ between the noncommutative $L^2$-space and the extended anyon Fock space. Under the unitary $U$, each operator $\la\omega,\varphi\ra$ takes a Jacobi  form in the extended anyon Fock space. We  show that this Jacobi field has the simplest form (in a sense) when  $\nu$ is the same measure as in Theorem~\ref{ur7o67r6},   i.e., $\nu$ is Kolmogorov's measure  of a white noise measure $\mu$ from the Meixner class. Furthermore, in this  case, analogs of formulas \eqref{yufru7rf}--\eqref{ufgutfr8ub} hold.

Finally, Section~\ref{ur867} is devoted to the proofs of the main results.

Among numerous open problems regarding  the anyon Meixner-type  white noise, let us mention only  two:

(i) In both the classical and free cases, the generating functions of the Meixner-type orthogonal polynomials are explicitly known and play an important role in the studies of these polynomials. In the anyon case, the form of the generating function is not yet known, even in the Gaussian case. The main difficulty lies in the fact that both the classical and  free Meixner-type polynomials have corresponding systems of orthogonal polynomials on the real line. However, the anyon case is purely infinite dimensional and has no related one-dimensional theory.

(ii) As shown in   \cite{afs}, in the classical case,
the   L\'evy processes from the Meixner class with $\eta>0$ are related to the renormalized squares of boson white noise. Is it possible to interpret anyon Meixner-type  white noises as those related to renormalized squares of anyon white noise?

\section{Noncommutative L\'evy white noise for anyon statistics}\label{uiy976}
\subsection{Anyon Fock space and anyon commutation relations}
\label{hfgyufzua}

Let  $\mathcal B(X)$ denote the Borel $\sigma$-algebra on $X$, and let $\mathcal B_0(X)$ denote the family of  all sets from  $\mathcal B(X)$ which have compact closure. Let $m=m(dx)=dx$ denote  the Lebesgue measure on $(X,\mathcal B(X))$.

For each $n\ge2$, we define 
\begin{equation}\label{ydr5w54} X^{(n)}:=\big\{(x_1,\dots,x_n)\in X^n\mid \forall 1\le i<j\le n:\  x_i\ne x_{j}\big\}.\end{equation}
Since the measure $m$ is non-atomic,
\begin{equation}\label{rtsew6uwy}m^{\otimes n}(X\setminus X^{(n)})=0.\end{equation}

 We introduce a strict total order on $X$ as follows, for any $x=(x^1,\dots,x^d),y=(y^1,\dots,y^d)\in X$, $x\ne y$, we set $x<y$ if for some $j\in\{1,\dots,d\}$, we have $x^1=y^1$,\dots, $x^{j-1}=y^{j-1}$ and $x^j<y^j$.

We  fix a number $q\in\mathbb C$ with $|q|=1$, and define a function $Q:X^{(2)}\to\mathbb C$ as follows:
$$ Q(x,y)=\begin{cases}
q,&\text{if }x<y,\\
\bar q,&\text{if }y<x.
\end{cases}$$
Note that the function $Q$ is Hermitian:
\[
Q(x, y) = \overline{Q(y, x)},\quad (x, y)\in X^{(2)}.
 \]
A function $f^{(n)}:X^{(n)}\to\mathbb C$ ($n\ge2$) is called $Q$-symmetric if, for each $j=1,\dots,n-1$,
\begin{equation}\label{uyr57eses} f^{(n)}(x_1,\dots,x_n)=Q(x_j,x_{j+1})f^{(n)}(x_1,\dots,x_{j-1},x_{j+1},x_{j},x_{j+2},\dots,x_n).\end{equation}

 Let $\mathcal H:=L^2(X,m)$ be the  Hilbert space of all complex-valued, square-integrable  functions on $X$. Thus, for each $n\in\mathbb N$, $\mathcal H^{\otimes n}=L^2(X^n,m^{\otimes n})$. 
 In view of \eqref{rtsew6uwy}, we have $\mathcal H^{\otimes n}=L^2(X^{(n)},m^{\otimes n})$.  
 We define a complex Hilbert space  $\mathcal H^{\cd n}$ as the (closed) subspace of   $\mathcal H^{\otimes n}$ consisting of all ($m^{\otimes n}$-versions of) $Q$-symmetric functions in $\mathcal H^{\otimes n}$. Let  $\Sym_n$ denote the orthogonal projection of $\mathcal H^{\otimes n}$ onto $\mathcal H^{\cd n}$. This operator has the following explicit form: for each $f^{(n)}\in\mathcal H^{\otimes n}$,
\begin{multline}\label{oitr8o} (\Sym_nf^{(n)})(x_1,\dots,x_n)\\=\frac1{n!}\sum_{\pi\in \mathfrak S_n}
Q_\pi(x_1,\dots,x_n)f^{(n)}(x_{\pi^{-1}(1)},\dots,x_{\pi^{-1}(n)}),\quad (x_1,\dots,x_n)\in X^{(n)}.\end{multline}
Here $\mathfrak S_n$ denotes the group of all permutations of $1,\dots,n$
and
\begin{equation}\label{y7e57ie} Q_\pi(x_1,\dots,x_n):=\prod_{\substack{1\le i<j\le n\\ \pi(i)>\pi(j)}} Q(x_{i},x_j),\quad (x_1,\dots,x_n)\in X^{(n)}.\end{equation}

We can now define a $Q$-symmetric tensor product $\cd$. For any $m,n\in\mathbb N$ and any $f^{(m)}\in\mathcal H^{\cd m}$ and $g^{(n)}\in\mathcal H^{\cd n}$, we set
$f^{(m)}\cd g^{(n)}:=\Sym_{m+n}(f^{(m)}\otimes g^{(n)})$. Note that this tensor product is associative. Note also that, for $q=1$, $\cd$ is the usual symmetric tensor product, while for $q=-1$, $\cd$ is the usual antisymmetric tensor product.

We define an anyon   Fock space by
$$\mathcal F^Q(\mathcal H):=\bigoplus_{n=0}^\infty \mathcal H^{\cd n}n!\, .$$
Thus, $\mathcal F^Q(\mathcal H)$ is the Hilbert space which consists of all sequences $F=(f^{(0)}, f^{(1)},f^{(2)},\dots)$ with $f^{(n)}\in  \mathcal H^{\cd n}$ ($\mathcal H^{\cd 0}:=\mathbb C$) satisfying
$$\|F\|^2_{\mathcal F^Q(\mathcal H)}:=\sum_{n=0}^\infty\|f^{(n)}\|^2_{\mathcal H^{\cd n}}n!<\infty.$$
(The inner product in $\mathcal F^Q(\mathcal H)$ is  induced by the norm in this space.)
The vector $\Omega:=(1,0,0,\dots)\in \mathcal F^Q(\mathcal H)$ is called the vacuum.
We  denote by $\mathcal F_{\mathrm{fin}}^Q(\mathcal H)$ the subspace of $\mathcal F^Q(\mathcal H)$ consisting of all finite sequences $$F=(f^{(0)},f^{(1)},\dots,f^{(n)},0,0,\dots)$$ in which $f^{(i)}\in \mathcal H^{\cd i}$ for $i=0,1,\dots,n$, $n\in\mathbb N$. This space can be endowed with the topology of the topological direct sum of the $\mathcal H^{\cd n}$ spaces. Thus, convergence in $\mathcal F_{\mathrm{fin}}^Q(\mathcal H)$ means uniform finiteness of non-zero components and coordinate-wise convergence in $\mathcal H^{\cd n}$.

For each $h\in\mathcal H$, we define a  creation operator $a^+(h)$ and an annihilation operator $a^-(h)$ as the linear operators acting on $\mathcal F_{\mathrm{fin}}^Q(\mathcal H)$ given by
$$
a^{+}(h)f^{(n)} := h \cd f^{(n)},\quad f^{(n)}\in \mathcal H^{\cd n},\quad a^{-}(h):= a^{+}(h)^*\restriction _{\mathcal F_{\mathrm{fin}}^Q(\mathcal H)}.
$$
Both $a^+(h)$ and $a^-(h)$ act continuously on $\mathcal F_{\mathrm{fin}}^Q(\mathcal H)$. In fact, for any $h\in\mathcal H$ and $f^{(n)}\in\mathcal H^{\cd n}$, we have
\begin{align}
& (a^+(h)f^{(n)})(x_1,\dots,x_{n+1})=\frac1{n+1}\Big[
h(x_1)f^{(n)}(x_2,\dots,x_{n+1})\notag\\
&\quad +\sum_{k=2}^{n+1}Q(x_1,x_k)Q(x_2,x_k)\dotsm Q(x_{k-1},x_k)h(x_k)f^{(n)}(x_1,\dots,x_{k-1},x_{k+1},\dots,x_{n+1})\Big],\notag\\
& (a^-(h)f^{(n)})(x_1, \dots , x_{n-1}) =
 n\int_{X}\overline{h(y)}\,f^{(n)}(y, x_1, \dots ,
x_{n-1})\,dy.\label{ufcu}
\end{align}
The action of the annihilation operator can also be written in the following form: for any $h\in\mathcal H$ and $f^{(n)}\in\mathcal H^{\otimes n}$,
\begin{multline} (a^-(h) \Sym_n f^{(n)})(x_1,\dots,x_{n-1})
=\Sym_{n-1}\bigg(\int_X\overline{h(y)}\bigg[\sum_{k=1}^n
Q(y,x_1)Q(y,x_2)\\
\times\dotsm \times Q(y,x_{k-1}) f^{(n)}
(x_1,x_2,\dots,x_{k-1},y,x_k,\dots,x_{n-1})
\bigg]dy\bigg).\label{hyfd7urf}\end{multline}

Let us now discuss the creation and annihilation operators at points of the space $X$.
At least informally, for each $x\in X$, we may consider a delta function at $x$, denoted by $\delta_x$.
Then we can heuristically define $\partial_x^\dag:=a^+(\delta_x)$ and $\partial_x:=a^-(\delta_x)$, so that
\begin{equation}\label{bgyiugtf8utfg}
\partial_x^\dag f^{(n)}=\delta_x\cd f^{(n)},\quad \partial_x f^{(n)}:=nf^{(n)}(x,\cdot).\end{equation}
Thus,
\begin{equation}\label{out979p}
a^+(h):=\int_X dx\, h(x)\partial_x^\dag\,,\quad
a^-(h)=\int_X  dx\, \overline{h(x)}\,\partial_x\,.
\end{equation}
Note that the second formula in  \eqref{bgyiugtf8utfg} is a rigorous definition of $\di_x$ (for $m$-a.a.\ $x\in X$), while the first formula in \eqref{out979p} is  the rigorous definition of the integral $\int_X dx\, h(x)\partial_x^\dag$.

 Let $B_0(X^n)$ denote the space of all complex-valued bounded measurable functions on $X^n$ with compact support. Let $g^{(n)}\in B_0(X^n)$.
Fix any sequence of $+$ and $-$ of length $n\ge 2$, and denote
it by $(\sharp_1, \dots , \sharp_n)$. It is easy to see that the expression
\[
\int_{X^n}dx_1\dotsm dx_n\,g^{(n)}(x_1, \dots ,
x_n)\partial^{\sharp_1}_{x_1}\dotsm \partial^{\sharp_n}_{x_n}
\]
identifies a linear continuous operator on $\mathcal  F_{\mathrm{fin}}^Q(\mathcal H)$. Here
 we used the notation $\partial_x^+:=\partial_x^\dag$, $\partial^-_x:=\partial_x$.

The creation and annihilation operators satisfy the anyon commutation
relations:
\begin{align}
\partial_x\partial_y^\dag &= \delta(x, y)+Q(x, y)\partial^\dag_y\partial_x,\label{fdytde} \\
\partial_x\partial_y &= Q(y,x)\partial_y\partial_x ,\label{jhgyufd}\\
\partial^\dag_x\partial^\dag_y &= Q(y, x)\partial^\dag_y\partial^\dag_x .\label{pr++}
\end{align}
Here  $\delta(x, y)$ is understood as:
\[
\int_{X^2}dx\, dy\,g^{(2)}(x, y)\delta(x, y) :=
\int_{X}dx\, g^{(2)}(x, x).
\]
Formulas \eqref{fdytde}--\eqref{pr++} make  rigorous sense after smearing with   functions $g^{(2)}\in B_0(X^2)$.
Note that, for $q=1$, equations \eqref{fdytde}--\eqref{pr++} become the canonical commutation relations, while for $q=-1$ they become the canonical anticommutation relations.

\begin{remark}
Let $D:=\{(x,x)\mid x\in X\}$. Note that, for each $g^{(2)}\in B_0(X^2)$ which has support in $D$, the operator $\int_{X^2}dx\,dy\, g^{(2)}(x,y)\partial^\dag_y\partial_x$ is equal to zero. Hence, it does not influence \eqref{fdytde} that we have not identified the function $Q$ on $D$.
\end{remark}

For a bounded linear operator $A$ in $\mathcal H$, we define
the differential second quantization of $A$, denoted by $d\Gamma(A)$,
as a linear continuous operator on $\mathcal F_{\mathrm{fin}}^Q(\mathcal H)$ given by
$d\Gamma(A)\Omega:=0$ and
$$d\Gamma(A)\restriction \mathcal H^{\cd n}:=\Sym_n(A\otimes \mathbf 1\otimes\dots\otimes \mathbf 1+
\mathbf 1\otimes A\otimes \mathbf 1\otimes\dots\otimes\mathbf 1+\dots+\mathbf 1\otimes\dots\otimes\mathbf1\otimes A)$$
for each $n\in\mathbb N$. For each a.e.\ bounded function $h\in L^\infty(X,m)$, we define a neutral operator
\begin{equation}\label{ur75r7} a^0(h):=\int_X dx\,h(x)\partial_x^\dag\partial_x.\end{equation}
 According to formulas \eqref{bgyiugtf8utfg} and \eqref{out979p}, we have
 \begin{align}
 \big(a^0(h)f^{(n)}\big)(x_1,\dots,x_n)&=\bigg(\int_X dx\, h(x)\di_x^\dag f^{(n)}(x,\cdot)\bigg)(x_1,\dots,x_n)\notag\\
 &=n\operatorname{Sym}_n\big(h(x_1)f^{(n)}(x_1,x_2,\dots,x_n)\big).\label{jfiu}
 \end{align}
From here one easily gets
\begin{equation}\label{gtdytlut9t8ghogt}
(a^0(h)f^{(n)})(x_1,\dots,x_n)=\big(h(x_1)+\dots+h(x_n)\big)f^{(n)}(x_1,\dots,x_n).\end{equation}
Hence, $a^0(h)=d\Gamma(M_h)$, where $M_h$ is the operator of multiplication by $h$.

\subsection{Anyon L\'evy white noise and noncommutative orthogonal polynomials}\label{utfr57rtf}
Let us now recall the construction of a L\'evy white noise over $X$ for anyon statistics, see \cite{BLW}.
Let $\nu$ be a probability measure on $(\mathbb{R},\mathcal B(\R))$. (In fact, we can  instead assume that $\nu$ is a finite measure. The results below will then require a trivial modification.)
We denote by $\mathscr{P}(\R)$ the linear space of polynomials on $\mathbb R$. We  assume that $\mathscr P(\R)$ is a dense subset of $L^2(\R,\nu)$. Note that the latter assumption is satisfied if, for example, \eqref{ft7er7i57} holds.

 We extend the function $Q$  by setting
$$Q(x_1,s_1,x_2,s_2):=Q(x_1,x_2),\quad (x_1,x_2)\in X^{(2)},\ (s_1,s_2)\in\mathbb R^2.$$
Thus, the value of the function $Q$ does not depend on $s_1$ and $s_2$.
Analogously to \eqref{uyr57eses}, we define the notion of a $Q$-symmetric function $f^{(n)}$ defined on the set
$$\left\{(x_1,s_1,\dots,x_n,s_n)\in (X\times\R)^n\mid (x_1,\dots,x_n)\in X^{(n)}\right\}.$$
For example, for $n=2$, the $Q$-symmetry means:
$$ f^{(2)}(x_1,s_1,x_2,s_2)=Q(x_1,x_2)f^{(2)}(x_2,s_2,x_1,s_1).$$

We   next set $$\mathcal G:=L^2(X\times \mathbb R,m\otimes \nu)=\h\otimes L^2(\mathbb R,\nu),$$
 and consider the corresponding $Q$-symmetric Fock space $\mathcal F^Q(\mathcal G)$, which is constructed by analogy with  $\mathcal F^Q(\mathcal H)$.
  Let $\mathcal F_{\mathrm{fin}}^Q(\h\otimes\mathscr P(\mathbb R))$ denote the linear subspace  of $\mathcal F^Q(\mathcal G)$ which consists of all finite sequences
$$F=(F^{(0)},F^{(1)},\dots,F^{(n)},0,0,\dots),\quad n\in\mathbb N_0,$$
such that each $F^{(k)}$ with $k\ne0$ has the form
$$F^{(k)}(x_1,s_1,\dots,x_k,s_k)=\Sym_k\left[\sum_{(i_1,i_2,\dots,i_k)\in\{0,1,\dots,N\}^k}f_{(i_1,i_2,\dots,i_k)}
(x_1,x_2,\dots,x_k)s_1^{i_1}s_2^{i_2}\dotsm s_{k}^{i_k}\right],$$
where $f_{(i_1,i_2,\dots,i_k)}\in \h^{\otimes k}$ and $N\in\mathbb N$. Clearly, $\mathcal F_{\mathrm{fin}}^Q(\h\otimes\mathscr P(\mathbb R))$ is dense in $\mathcal F^Q(\mathcal G)$.

  We denote $1(s):=1$ and $\id(s):=s$ for $s\in\R$. Thus, $1,\id\in\mathscr P(\R)$. We denote by $C_0(X\mapsto\mathbb R)$ the space of all real-valued continuous  functions on $X$ with compact support. 
  For each $f\in C_0(X\mapsto\mathbb R)$, we define an operator
\begin{equation}\label{g7r75e}\la \omega,f\ra:=a^{+}(f\otimes 1)+a^0(f\otimes \id)+a^-(f\otimes 1)\end{equation}
in $\mathcal F^Q(\mathcal G)$ with domain $\mathcal F_{\mathrm{fin}}^Q(\h\otimes\mathscr P(\mathbb R))$. Clearly, each operator $\la \omega,f\ra$ maps \linebreak  $\mathcal F_{\mathrm{fin}}^Q(\h\otimes\mathscr P(\mathbb R))$ into itself. In fact, under  assumption \eqref{ft7er7i57}, each  $F\in\mathcal F_{\mathrm{fin}}^Q(\h\otimes\mathscr P(\mathbb R))$ is an analytic vector for each operator $\la \omega,f\ra$ with  $f\in C_0(X\mapsto\mathbb R)$, which implies that the operators $\la \omega,f\ra$  are essentially self-adjoint on $\mathcal F_{\mathrm{fin}}^Q(\h\otimes\mathscr P(\mathbb R))$ (see e.g.\ \cite[Sec.~X.2]{RS2}).

\begin{remark}
Let us keep the notation $\la \omega,f\ra$ for the closure of this operator in $\mathcal F^Q(\mathcal G)$. Thus the operators $\la \omega,f\ra$ are self-adjoint. 
In the boson case, $q=1$, these operators also commute in the sense of commutation of their resolutions of the identity. By using e.g.\ the projection spectral theorem \cite{BK}, one  shows \cite{DL} that there exists a unitary isomorphism between the symmetric Fock space 
$\mathcal F^Q(\mathcal G)$ and the space $L^2(\mathscr D'(X),\mu)$, where $\mu$ is the L\'evy white noise measure with Fourier transform \eqref{jig8yugtf8}. 
Under this unitary isomorphism, the vacuum vector $\Omega$ becomes the constant function $1$, and each  operator $\la \omega,f\ra$
becomes the operator of multiplication by the random variable $\la \omega,f\ra$ in 
$L^2(\mathscr D'(X),\mu)$. In other words, $\mu$ is the spectral measure of the family of commuting self-adjoint operators $\big(\la \omega,f\ra\big)_{f\in C_0(X\mapsto\R)}$. 
In particular, the operators $\big(\la \omega,f\ra\big)_{f\in C_0(X\mapsto\R)}$ in the symmetric Fock space $\mathcal F^Q(\mathcal G)$ can indeed be thought of as a L\'evy white noise. Let us also note that the unitary operator between $\mathcal F^Q(\mathcal G)$ and  $L^2(\mathscr D'(X),\mu)$ was originally derived by It\^o, by using multiple stochastic integrals, see 
\cite{Ito}. 
\end{remark}

\begin{remark}
Note that, if the measure $\nu$ is concentrated at one point, $\lambda\in\mathbb R$, then
$\mathcal G=\h$ and each operator $\la \omega,f\ra$ 
has the following form in $\mathcal F^Q(\mathcal H)$:
\begin{equation}\label{vgfi} \la \omega,f\ra:=a^+(f)+a^-(f)+\lambda a^0(f).\end{equation}
The choice $\lambda=0$ corresponds to an anyon Gaussian white noise, while $\lambda\ne0$ corresponds to an anyon centered white noise.  If we denote
\begin{equation}\label{vcydcyd}\omega(x):=\partial_x^\dag+\lambda\partial_x^\dag\partial_x+\partial_x,\quad x\in X,\end{equation}
then, by \eqref{out979p}, \eqref{ur75r7}, \eqref{vgfi}, and \eqref{vcydcyd}, we get
$$ \la 
\omega,f\ra=\int_X dx\, \omega(x)f(x),\quad f\in C_0(X\mapsto\R),$$ which justifies the notation $\la \omega,f\ra$.
Thus, $(\omega(x))_{x\in X}$ is the anyon Gaussian/Poisson white noise.
Note that $\omega(x)$ is informally treated as an operator-valued distribution.
\end{remark}

We further denote by $C_0(X)$ the space of all complex-valued, continuous  functions on $X$ with compact support. For $f\in C_0(X)$, we set $\la \omega,f\ra:=\la
\omega,\Re f\ra+i\la \omega,\Im f\ra$.

Let $\mathscr P$ denote the complex unital $*$-algebra generated by $(\la \omega,f\ra)_{f\in C_0(X)}$, i.e., the algebra of noncommutative polynomials in  variables $\la \omega,f\ra$. In particular, elements of $\mathscr P$ are linear operators acting on $\mathcal F_{\mathrm{fin}}^Q(\h\otimes\mathscr P(\mathbb R))$, and for each $p\in\mathscr P$, $p^*$ is the adjoint operator of $p$ in $\mathcal F^Q(\mathcal G)$, restricted to $\mathcal F_{\mathrm{fin}}^Q(\h\otimes\mathscr P(\mathbb R))$.

We define a vacuum state on $\mathscr P$ by
$$\tau(p):=(p\Omega,\Omega)_{\mathcal F^Q(\mathcal G)},\quad p\in\mathscr P.$$ We introduce a scalar product on $\mathscr P$
by
$$(p_1,p_2)_{L^2(\tau)}:=\tau(p_2^*p_1)=(p_1\Omega,p_2\Omega)_{\mathcal F^Q(\mathcal G)}, \quad p_1,p_2\in\mathscr P.$$
Let
\begin{equation}\label{hufdy676}
\widetilde {\mathscr P}:=\{p\in\mathscr P\mid (p,p)_{L^2(\tau)}=0\},
 \end{equation}
 and define the noncommutative $L^2$-space $L^2(\tau)$ as the completion  of the quotient space $\mathscr P/\widetilde{\mathscr P}$ with respect to the norm generated by the scalar product $(\cdot,\cdot)_{L^2(\tau)}$. Elements $p\in\mathscr P$ are treated as representatives of the equivalence classes from
$\mathscr P/\widetilde{\mathscr P}$, and so $\mathscr P$ becomes a dense subspace of $L^2(\tau)$. As shown in \cite{BLW}, the vacuum vector $\Omega$ is cyclic for the family of operators $(\la \omega,f\ra)_{f\in C_0(X\mapsto\mathbb R)}$.
Consider a linear mapping $I:\mathscr P\to
\mathcal F^Q(\mathcal G)$ defined by 
$$Ip:=p\Omega\quad \text{for $p\in\mathscr P$}.$$
 Then $Ip_1=Ip_2$ if $p_1,p_2\in\mathscr P$ are such that   $p_1-p_2\in\widetilde{\mathscr P}$, and   $I$ extends to a unitary operator $I:L^2(\tau)\to\mathcal F^Q(\mathcal G)$.

Note that, for each $p\in\mathscr P$ and $f\in C_0(X)$,
\begin{equation}\label{tyr75}
 I\big(\la \omega,f\ra p\big)=\la \omega,f\ra  (Ip),\end{equation}
i.e., under the unitary $I$, the operator of left multiplication by $\la \omega,f\ra$ in $L^2(\tau)$ becomes  the operator $\la \omega,f\ra$ in $\mathcal F^Q(\mathcal G)$.

Let us consider the  topology on $C_0(X)$
which yields the following notion of convergence: $f_n\to f$ as $n\to\infty$ means that there exists a set $\Delta\in\mathcal B_0(X)$ such that
$\operatorname{supp}(f_n)\subset\Delta$ for all $n\in\mathbb N$ and
\begin{equation}\label{ilyufre75iei7}\sup_{x\in X}|f_n(x)-f(x)|\to0\quad\text{as }n\to\infty.\end{equation}
By linearity and continuity we can extend the mapping
$$C_0(X)^n\ni (f_1,\dots,f_n)\mapsto\la \omega^{\otimes n}, f_1\otimes\dots\otimes f_n\ra
=\la \omega, f_1\ra\dotsm \la \omega, f_n\ra\in\mathscr P
$$
   to a mapping
$$ C_0(X^n)\ni f^{(n)}\mapsto \la \omega^{\otimes n},f^{(n)}\ra \in L^2(\tau),$$
and $\la \omega^{\otimes n},f^{(n)}\ra $ can be thought of as a linear operator acting in $\mathcal F_{\mathrm{fin}}^Q(\h\otimes\mathscr P(\mathbb R))$. We will think of $\la \omega^{\otimes n},f^{(n)}\ra $ as a continuous monomial of order $n$. Sums of such operators and (complex) constants form the set $\mathscr{CP}$ of continuous polynomials (of $\omega$). Evidently, $\mathscr P\subset\mathscr{CP}$.

Completely analogously to \eqref{gyut8t}, we derive the orthogonal decomposition
\begin{equation}\label{tyr65}
L^2(\tau)=\bigoplus_{n=0}^\infty \mathscr{OP}_n\end{equation}
(we used obvious notations). For any $f^{(n)}\in C_0(X^n)$, we denote by $\la P_n(\omega),f^{(n)}\ra$ the orthogonal projection of $\la \omega^{\otimes n},f^{(n)}\ra$ onto $ \mathscr{OP}_n$. The set of finite linear sums of $\la P_n(\omega),f^{(n)}\ra$ and (complex) constants is denoted by $\mathscr{OCP}$ (orthogonalized continuous polynomials).

\begin{remark}
Note that $\la P_1(\omega),f\ra=\la \omega,f\ra$.
\end{remark}

\begin{remark}
Note that, in subsec.~\ref{fytry}, we used functions $f^{(n)}\in\mathscr D(X)^{\otimes n}$ when defining $\mathscr{CP}$ and $\mathscr{OCP}$, while now we are using $f^{(n)}\in C_0(X^n)$ to define $\mathscr{CP}$ and $\mathscr{OCP}$. The reason is that, in the noncommutative setting, there is no need for $f^{(n)}$ to be smooth, while in the classical case, $q=1$, Theorem~\ref{ur7o67r6} still holds for the sets $\mathscr{CP}$ and $\mathscr{OCP}$ as defined in this section.
\end{remark}

\section{Main results}\label{yre6tr5}

\subsection{The measures $m_\nu^{(n)}$}\label{utf7r5svgy}

Let $(p_k)_{k=0}^\infty$ denote the system of monic orthogonal polynomials in $L^2(\mathbb R,\nu)$. (If the support of $\nu$ is finite and consists of $N$ points, we set $p_k:=0$ for $k\ge N$.) Hence, $(p_k)_{k=0}^\infty$ satisfy  the recursion formula
\begin{equation}\label{hdtrss}sp_k(s)=p_{k+1}(s)+b_kp_{k}(s)+a_kp_{k-1}(s),\quad k\in\mathbb N_0,\end{equation}
with $p_{-1}(s):=0$, $a_k>0$, and $b_k\in\mathbb R$. (If the support of $\nu$ has $N$ points, $a_k=0$ for $k\ge N$.)

We define
\begin{equation}\label{yuft8uotfr8o}c_k:=a_0a_1\dotsm a_{k-1},\quad k\in\mathbb N,\end{equation}
where $a_0:=1$ and the $a_k$'s for $k\in\mathbb N$ are the coefficients from  formula \eqref{hdtrss}.
We equivalently have:
\begin{equation}\label{yufrur}c_k=\int_{\mathbb R}p_{k-1}(s)^2\,\nu(ds),\quad k\in\mathbb N,\end{equation}
which is a well known fact of the theory of orthogonal polynomials.
Note that $c_1=1$ and $c_k=0$ for $k\ge2$
 if and only if the measure $\nu$ is concentrated at one point.

We denote  by $\Pi(n)$ the set of all (unordered) partitions of the set $\{1,\dots,n\}$. For each partition
$\theta=\{\theta_1,\dots,\theta_l\}\in\Pi(n)$, we set
$|\theta|:=l$. For each  $\theta\in\Pi(n)$, we denote by $X^{(n)}_\theta$ the subset of $X^n$ which consists of all $(x_1,\dots,x_n)\in X^n$ such that, for all $1\le i<j\le n$, $x_i=x_j$ if and only if $i$ and $j$ belong to the same element of the partition $\theta$. Note that the sets $X^{(n)}_\theta$ with $\theta\in \Pi(n)$
form a partition of $X^n$. Note also that, by \eqref{ydr5w54}, $X^{(n)}=X_\theta^{(n)}$ for the minimal partition $\theta=\{\{1\},\,\{2\},\dots,\, \{n\}\}$.

Let us fix $n\in\mathbb N$,  a permutation $\pi\in \mathfrak S_n$,
and a partition $\theta=\{\theta_1,\dots,\theta_l\}\in\Pi(n)$ satisfying
\begin{equation}\label{ft6e}\max \theta_1<\max\theta_2<\dots<\max\theta_l.\end{equation}
We define a measure $m_{\nu,\,\theta}^{(n)}$ on $X_\theta^{(n)}$ as the push-forward of the measure
$$  \big(c_{|\theta_1|}\dotsm c_{|\theta_l|}\big)n!
\big(|\theta_1|!\dotsm |\theta_l|!\big)^{-1}\, m^{\otimes l}$$
on $X^{(l)}$ under the mapping
$$ X^{(l)}\ni y=(y_1,\dots,y_l)\mapsto (R_\theta^1y,\dots, R_\theta^n y)\in X_\theta^{(n)},$$
where $R_\theta^iy=y_j$ for $i\in\theta_j$. 
Here $|\theta_i|$ denotes the number of elements of the set $\theta_i$.
Recalling that the sets $X_\theta^{(n)}$ with $\theta\in\Pi(n)$ form a partition of $X^{n}$, we define a measure $m_\nu^{(n)}$ on $X^n$ such that the restriction of $m_\nu^{(n)}$ to each $X_\theta^{(n)}$ is equal to $m_{\nu,\,\theta}^{(n)}$. Note that the restriction of 
$m_\nu^{(n)}$ to $X^{(n)}$ is equal to $n!\, m^{\otimes n}$.

For example, for $n=2$, we get
\begin{align*}
\int_{X^2}f^{(2)}(x_1,x_2)\,m_\nu^{(2)}(dx_1\times dx_2)
&=\int_{\{x_1\ne x_2\}}f^{(2)}(x_1,x_2)\,dx_1\,dx_2\, 2+\int_X f^{(2)}(x,x)\, dx\, c_2\\
&=\int_{X^2}f^{(2)}(x_1,x_2)\,dx_1\,dx_2\, 2+\int_X f^{(2)}(x,x)\, dx\, c_2.
\end{align*}

\subsection{An extended anyon Fock space}\label{guyf7r}
Let us recall that, in subsec.~\ref{hfgyufzua}, see in particular \eqref{uyr57eses}, we defined the notion of a $Q$-symmetric function $f^{(n)}:X^{(n)}\to\mathbb C$. Our next  aim is to extend this notion to a complex-valued function 
defined on the whole $X^n$.

Let us fix  a permutation $\pi\in \mathfrak S_n$
and a partition $\theta=\{\theta_1,\dots,\theta_l\}\in\Pi(n)$ satisfying \eqref{ft6e}. The permutation $\pi$ maps the  partition $\theta$ into a new partition
$$\{\pi\theta_1,\dots,\pi\theta_l\}\in\Pi(n).$$
We call this new partition $\beta=\{\beta_1,\dots,\beta_l\}$, where the elements of the partition $\beta $ are enumerated in such a way that
\begin{equation}\label{adiov}\max\beta_1<\max\beta_2<\dots<\max\beta_l.\end{equation}
Thus, the permutation $\pi\in \mathfrak S_n$ identifies a permutation $\widehat \pi\in \mathfrak S_l$ (dependent on $\theta$) such that
\begin{equation}\label{huyfr7i5er}\pi\theta_i=\beta_{\widehat \pi(i)},\quad i=1,\dots,l.\end{equation}

Recall the complex-valued function $Q_\pi(x_1,\dots,x_n)$ on $X^{(n)}$ defined by \eqref{y7e57ie}. We will now extend this function to the whole set $X^{n}$ as follows. We fix any $\theta=\{\theta_1,\dots,\theta_l\}\in\Pi(n)$  satisfying  \eqref{ft6e} and any  $(x_1,\dots,x_n)\in X_\theta^{(n)}$. We
denote by $x_{\theta_1},x_{\theta_2},\dots,x_{\theta_l}$ the elements $x_{i_1},x_{i_2},\dots,x_{i_l}$ with $i_1\in\theta_1, i_2\in\theta_2,\dots,i_l\in\theta_l$, respectively.
We  set
\begin{equation}\label{uyr75rw}  Q _\pi(x_1,\dots,x_n):= \prod_{\substack{1\le i<j\le l\\ \widehat\pi(i)>\widehat \pi(j)}} Q(x_{\theta_i},x_{\theta_j}),\end{equation}
where the permutation $\widehat\pi\in\mathfrak S_l$ is as above.
Note that, for the  partition
$$\theta=\big\{\{1\},\{2\},\dots,\{n\}\big\},$$ the restriction of the function $ Q_\pi$ to the set $X_\theta^{(n)}=X^{(n)}$ is indeed equal to the function  $Q_\pi$  defined by \eqref{y7e57ie}.

We will say that a function $f^{(n)}:X^n\to\mathbb C$  is $Q$-symmetric if, for each permutation $\pi\in\mathfrak S_n$, 
\begin{equation}\label{ft7ier5ird} f^{(n)}(x_1,\dots,x_n)= Q_\pi(x_1,\dots,x_n)f^{(n)}(x_{\pi^{-1}(1)},\dots,x_{\pi^{-1}(n)}),\quad (x_1,\dots,x_n)\in X^n.\end{equation}
In particular, the restriction of such a function to $X^{(n)}$ is then $Q$-symmetric according to our definition in 
subsec.~\ref{hfgyufzua}, i.e., it satisfies \eqref{uyr57eses}.

Next, for a function $f^{(n)}:X^n\to\mathbb C$, we define
\begin{multline} (\Sym_n\, f^{(n)})(x_1,\dots,x_n)\\=\frac1{n!}\sum_{\pi\in \mathfrak S_n}
 Q_\pi(x_1,\dots,x_n)f^{(n)}(x_{\pi^{-1}(1)},\dots,x_{\pi^{-1}(n)}),\quad (x_1,\dots,x_n)\in X^n.\label{tr866}\end{multline}
Clearly, the restriction of the function  $\Sym_n\,f^{(n)}$
to the set $X^{(n)}$ is still given by  \eqref{oitr8o}.

We denote by $\mathbf F^Q_n(\mathcal H,\nu)$ the subspace of the complex $L^2$-space $L^2(X^n,m_\nu^{(n)})$ which consists of ($m_\nu^{(n)}$-versions of) $Q$-symmetric functions.

\begin{proposition}\label{yfd6rd6} For each $n\in\mathbb N$, $\Sym_n$ is the orthogonal projection of $L^2(X^n,m_\nu^{(n)})$ onto $\mathbf F^Q_n(\mathcal H,\nu)$.
\end{proposition}

We also set  $\mathbf F^{ Q}_{0}(\mathcal H,\nu)=\{c\Omega\mid c\in\mathbb C\}$, where $\Omega$ is the vacuum vector. We define an extended anyon Fock space
$$\mathbf F^{Q}(\mathcal H,\nu):=\bigoplus_{n=0}^\infty
\mathbf  F^{Q}_{n}(\mathcal H,\nu).$$

If  the measure $\nu$ is concentrated at one point (and so $c_1=1$ and $c_k=0$ for $k\ge2$),  we get $\mathbf F^{Q}(\mathcal H,\nu)=\mathcal F^Q(\mathcal H)$, i.e., $\mathbf F^{Q}(\mathcal H,\nu)$ is the usual anyon Fock space. Otherwise,   $\mathcal F^Q(\mathcal H)$ is a proper subspace of $\mathbf F^{\mathbf Q}(\mathcal H,\nu)$. Indeed, recalling formula \eqref{rtsew6uwy}, we may  embed  $\mathcal F^Q(\mathcal H)$ into $\mathbf F^{Q}(\mathcal H,\nu)$ by identifying  each function  $f^{(n)}\in\mathcal H^{\cd n}$ with the function from $\mathbf F^{Q}_{ n}(\mathcal H,\nu)$ which is equal to $f^{(n)}$ on $X^{(n)}$, and to 0 otherwise. Evidently, the orthogonal complement to $\mathcal F^Q(\mathcal H)$ in $\mathbf F^{\mathbf Q}(\mathcal H,\nu)$ is a non-zero space in this case.

Using the orthogonal decomposition \eqref{tyr65}, we will now construct a unitary isomorphism between $L^2(\tau)$ and the extended anyon Fock space $\mathbf F^{\mathbf Q}(\mathcal H,\nu)$.

\begin{theorem}\label{utu8}
Let $f^{(n)},g^{(n)}\in C_0(X^n)$. Then
\begin{equation}\label{gilyr7e5if}
\big(\la P_n(\omega),f^{(n)}\ra , \,\la P_n(\omega),g^{(n)}\ra\big)_{L^2(\tau)}=(\Sym_n\,f^{(n)},\Sym_n\, g^{(n)})_{\mathbf F_{ n}^{Q}(\mathcal H,\nu)}.
\end{equation}
\end{theorem}

Since the set $C_0(X^n)$ is dense in  $L^2(X^n,m_\nu^{(n)})$, Theorem~\ref{utu8} implies that we can extended the mapping 
$$C_0(X^n)\ni f^{(n)}\mapsto \la P_n(\omega),f^{(n)}\ra\in L^2(\tau)$$
to a linear continuous operator 
$$L^2(X^n,m_\nu^{(n)})\ni f^{(n)}\mapsto \la P_n(\omega),f^{(n)}\ra\in L^2(\tau).$$ 
Note that, by Theorem~\ref{utu8}, for each $f^{(n)}\in L^2(X^n,m_\nu^{(n)})$,
$$\la P_n(\omega),f^{(n)}\ra=\la P_n(\omega),\Sym_nf^{(n)}\ra.$$

Thus, Theorem \ref{utu8} immediately implies

\begin{corollary}\label{igt87}
We have a unitary isomorphism
\begin{equation}\label{yut8t}
\mathbf F^{Q}(\mathcal H,\nu)\ni (f^{(n)})_{n=0}^\infty
\mapsto f^{(0)}+\sum_{n=1}^\infty \la P_n(\omega), f^{(n)}\ra\in L^2(\tau).\end{equation}
\end{corollary}

We denote the inverse of the unitary operator in \eqref{yut8t} by $U$. Thus, $U:L^2(\tau)\to \mathbf F^{Q}(\mathcal H,\nu)$ is a unitary operator; compare with Theorem~\ref{uyt8t8} (ii) in the boson case, $q=1$.

\subsection{Anyon L\'evy white noise as a Jacobi field}

In view of subsec.~\ref{utfr57rtf} and Corollary~\ref{igt87}, we have the following chain of unitary operators:
$$ \mathbf F^{Q}(\mathcal H,\nu)\overset {\text{}\ U}{\leftarrow} L^2(\tau)\overset I\to \mathcal F^Q(\mathcal G).$$
We also define a unitary operator 
$$\mathbf U:\mathbf F^{Q}(\mathcal H,\nu)\to \mathcal F^Q(\mathcal G),\quad \mathbf U:=IU^{-1}.$$

Let $h\in C_0(X)$. Recall formula \eqref{tyr75}, which says that, under $I^{-1}$, the operator $\la\omega,h\ra$ in 
$\mathcal F^Q(\mathcal G) $ becomes the operator of left multiplication by $\la\omega,h\ra$ in $L^2(\tau)$.  
We denote 
\begin{equation}\label{utr7r}
\mathbf J(h):=\mathbf U^{-1}\la\omega,h\ra \mathbf U.
\end{equation}
 Obviously, the operators $\mathbf J(h)$ form a Jacobi field  in the extended anyon Fock space $\mathbf F^{Q}(\mathcal H,\nu)$, i.e.,  each operator $\mathbf J(h)$ has a representation
\begin{equation}\label{iut78}
\mathbf J(h)=\mathbf J^+(h)+\mathbf J^0(f)+\mathbf J^-(h),\end{equation}
where $\mathbf J^+(h)$ is a creation operator, $\mathbf J^0(h)$ is a neutral operator, and $\mathbf J^-(h)$ is an annihilation operator. Equivalently, we have
$$\la \omega,h\ra\la P_n(\omega),f^{(n)}\ra=\la P_{n+1}(\omega),\mathbf J^+(h) f^{(n)}\ra +\la   P_{n}(\omega),\mathbf J^0(h) f^{(n)}\ra
+\la P_{n-1}(\omega),\mathbf J^-(h) f^{(n)}\ra. $$
Our next aim is to explicitly calculate the operators  $\mathbf J^\sharp(h)$, $\sharp=+,0,-$.

We define a linear space  $\mathcal F_{\mathrm{fin}}(B_0(X))$ of all finite vectors $(f^{(0)},f^{(1)},\dots,f^{(n)},0,0,\dots)$, where $f^{(0)}\in\mathbb C$, $f^{(i)}\in B_0(X^i)$, $i\ge1$. Evidently, the vacuum vector, $\Omega$, belongs to $\mathcal F_{\mathrm{fin}}(B_0(X))$.

For each $h\in C_0(X)$, we define   a neutral operator $\mathscr J^0(h)$ and an annihilation operator $\mathscr J_1^-(h)$
acting on $\mathcal F_{\mathrm{fin}}(B_0(X))$ as follows. We first set
\begin{equation}\label{uiti}
\mathscr J^0(h)\Omega=
 \mathscr J_1^-(h)\Omega:=0.\end{equation}
  Next,
 \begin{equation}\label{ir688u}\mathscr (\mathscr J^0(h)f^{(n)} )(x_1,\dots,x_n):=
\sum_{i=1}^n h(x_i)f^{(n)}(x_1,\dots,x_n)R^{(n)}_i(x_1,\dots,x_n).\end{equation}
Here, for each $\theta=\{\theta_1,\dots,\theta_l\}\in\Pi(n)$, the restriction of the function $R^{(n)}_i:X^n\to\R$ to the set $X_\theta^{(n)}$ is given by
\begin{equation}\label{ho9t97tg}R_i^{(n)}\restriction  X_\theta^{(n)} :=b_{\gamma(i,\theta)-1}\,/\gamma(i,\theta)\end{equation}
 In formula \eqref{ho9t97tg},  $\gamma(i,\theta):=|\theta_u|$ with
  $\theta_u\in\theta$ being chosen so that $i\in \theta_u$, and $(b_k)_{k=0}^\infty$ are the coefficients from \eqref{hdtrss}. Finally,
 \begin{multline}\label{igftr7y}(\mathscr J_1^-(h)f^{(n)})(x_1,\dots,x_{n-1})
 \\:=
\sum_{1\le i<j\le n}h(x_{j-1})f^{(n)}(x_1,\dots,x_{i-1},
\underbrace{x_{j-1}}_{\text{$i$-th place}},x_i,x_{i+1},\dots,\underbrace{x_{j-1}}_{\text{$j$-th place}},\dots,x_{n-1})\\
\times
S^{(n)}_{j-1}(x_1,\dots,x_{n-1}),\end{multline}
 where for any $\theta\in\Pi(n-1)$
\begin{equation}\label{oduhuaijo}S_{j-1}^{(n)}\restriction  X_\theta^{(n-1)} :=\frac{
2a_{\gamma(j-1,\theta)}}{\gamma(j-1,\theta)(\gamma(j-1,\theta)+1)}\,.\end{equation}
 Here  $(a_k)_{k=1}^\infty$ are also the coefficients from \eqref{hdtrss}. 

We define 
$$ \mathbf F_{\mathrm{fin}}^{Q}(B_0(X)):=\Sym\, \mathcal F_{\mathrm{fin}}(B_0(X)),$$
where $\Sym$ is the linear operator on $\mathcal F_{\mathrm{fin}}(B_0(X))$ satisfying $\Sym f^{(n)}:=\Sym_n f^{(n)}$ for $f^{(n)}\in B_0(X^n)$. We also denote $\mathbf B_0^Q(X^n):=\Sym_n B_0(X^n)$.

On $\mathbf F_{\mathrm{fin}}^{Q}(B_0(X))$, we define a $Q$-symmetric tensor product by setting, for any $f^{(m)}\in \mathbf B_0^Q(X^m)$, $g^{(m)}\in \mathbf B_0^Q(X^n)$,
\begin{equation}\label{vytfr7i6ed}
f^{(m)} \cd g^{(n)}:=\Sym_{m+n} (f^{(m)} \otimes g^{(n)}), \end{equation}
and extending it by linearity. 
Here $f^{(m)}\otimes g^{(n)}\in B_0(X^{m+n})$ is given by 
$$ (f^{(m)}\otimes g^{(n)})(x_1,\dots,x_{m+n})=f^{(m)}(x_1,\dots,x_m)g^{(n)}(x_{m+1},\dots,x_{m+n}).$$
We will prove below that the tensor product $\cd$ is associative. Furthermore, the restriction of $f^{(m)} \cd g^{(n)}$ to $X^{(m+m)}$ evidently coincides with $f^{(m)} \cd g^{(n)}$ as defined in subsec.~\ref{hfgyufzua}.

\begin{theorem}\label{fu7r7} For each $h\in C_0(X)$, 
$\mathbf J(h)$ is a linear operator on  $\mathbf F_{\mathrm{fin}}^{ Q}(B_0(X))$ which has representation 
\eqref{iut78}. 
For each $F\in \mathbf F_{\mathrm{fin}}^{Q}(B_0(X))$, we have
\begin{align*}
\mathbf J^+(h) F&=h\cd F,\\
\mathbf J^0(h)F&=\Sym(\mathscr J^0(h)F),
\end{align*}
and 
\begin{equation}\label{yur75}\mathbf J^-(h)=\mathbf J_1^-(h)+\mathbf J_2^-(h).
\end{equation}
 Here, 
$$\mathbf J^-_1(h) F=\Sym (\mathscr J_1^-(h)F)$$
and for each $f^{(n)}\in \mathbf B_0^Q(X^n)$ 
\begin{equation}
(\mathbf J_2^-(h)f^{(n)})(x_1,\dots,x_{n-1})=n\int_X dy\, h(y)f^{(n)}(y,x_1,\dots,x_{n-1}).\label{ufcxyf}\end{equation}
\end{theorem}

\subsection{A characterization of Meixner-type polynomials}
Recall that the operators $\mathscr J^0(h)$ and $\mathscr J^-(h)$ were defined by using the coefficients of the recursion relation \eqref{hdtrss} (i.e., by the measure $\nu$), and these operators do not depend on the type of anyon statistics, i.e., they are independent of $Q$.   

Recall the set of orthogonalized continuous polynomials, $\mathscr{OCP}$, defined in subsec.~\ref{utfr57rtf}. Let us consider the following condition.   

\begin{enumerate}
\item[(C)] For each $h\in C_0(X\mapsto\R)$, the linear operators  $\mathbf J^0(h)$ and   $\mathbf J^-_1(h)$ map the set $\mathscr{OCP}$ into itself.
\end{enumerate}

\begin{theorem}\label{urr8r} Assume that either $q\ne- 1$ or $q=-1$ and the support of the measure $\nu$ does not consist of exactly two points. Then condition {\rm (C)} is satisfied if and only if there exist constants $\lambda\in\mathbb R$ and $\eta\ge0$ such that the coefficients $a_k$, $b_k$ in the recursion formula \eqref{hdtrss} are given by
\begin{equation}\label{vggyufd7u}a_k=\eta k(k+1)\quad
(k\in\mathbb N),\quad
b_k=\lambda(k+1)\quad (k\in\mathbb N_0).\end{equation}
In the latter case, for any $h, f_1,\dots,f_n\in C_0(X)$, we have
\begin{align}
&\mathbf J(h) f_1\cd\dots\cd f_n=
h\cd f_1\cd\dots\cd f_n\notag\\
&\quad\text{}+\lambda
 \sum_{i=1}^n  f_1\cd \dots\cd f_{i-1}\cd (hf_i)\cd f_{i+1}\cd \dots\diamond f_n\notag\\
&\quad\text{}+2\eta\sum_{1\le i<j\le n} f_1\cd \dots\cd
f_{i-1}\cd f_{i+1}\cd \dots \cd f_{j-1}\cd
 (hf_if_j)\cd f_{j+1}\cd\dots\cd f_n\notag\\
&\quad\text{} + n \int_X dy\, h(y)(f_1\cd\dots\cd f_n)(y,\cdot).\label{yufu7edseaa}
  \end{align}
\end{theorem}

We see that, in the classical case, $q=1$, Theorem~\ref{urr8r} gives exactly the Meixner class of infinite dimensional polynomials, discussed in subsec.~\ref{fytry}. Note that the obtained class of the measures $\nu$ is independent on $q$. So, for such a choice of $\nu$, we  call  $\big(\la P_n(\omega),f^{(n)}\ra\big)$ a Meixner-type system of orthogonal (noncommutative) polynomials for anyon statistics.

\begin{remark}
In the fermion case ($q=-1$), if the support of the measure $\nu$ consists of  exactly two points, we could not 
prove that condition (C) always fails, but we conjecture this indeed to be the case.
\end{remark}

The following result can be easily proven.

\begin{proposition}\label{ytfr7} For each $q\in\mathbb C$, $|q|=1$ we have  equality $\mathscr{CP}=\mathscr{OCP}$ in the anyon Gaussian/Poisson case, i.e., when formula \eqref{vggyufd7u} holds with $\lambda\in\R$ and $\eta=0$.
\end{proposition}

However, due to the form of the operator $\mathbf J^-_2(h)$, see \eqref{ufcxyf}, equality $\mathscr{CP}=\mathscr{OCP}$  fails if $q\ne1$ and the the measure $\nu$ is not concentrated at one point. Still, in the classical case, $q=1$, 
Theorem~\ref{urr8r}  implies Theorem~\ref{ur7o67r6}.

\subsection{Anyon Meixner-type white noise}\label{ufuktydr}

We will assume in this subsection that  \eqref{vggyufd7u}  holds.
We may, at least informally, define,
$$\omega(x)=\la\omega,\delta_x\ra,\quad x\in X,$$
 so that for $h\in C_0(X)$,
\begin{equation}\label{yf7r}
\la \omega,h\ra=\int_X dx\, \omega(x)h(x).\end{equation}
Hence, we may think of $(\omega(x))_{x\in X}$
as an anyon Meixner-type white noise.

For $x\in X$, we define an annihilation operator $\partial_x$ as the linear operator acting on
$\mathbf F^{Q}_{\mathrm{fin}}(B_0(X))$ by the formula:
\begin{equation}\label{kbhfdvyd}(\partial_x f^{(n)})(x_1,\dots,x_{n-1}):= n f^{(n)}(x,x_1,\dots,x_{n-1}),\quad
(x_1,\dots,x_{n-1})\in X^{n-1},
\end{equation}
for $f^{(n)}\in \mathbf B_0^{Q}(X^n)$. Then, by \eqref{ufcxyf}, for $h\in C_0(X)$, we may interpret the operator $\mathbf J_2^-(h)$ as the integral
\begin{equation}\label{adiohdsrtdv}  \mathbf J_2^-(h)=\int_X dx\, h(x)\partial_x\,. \end{equation}

Next, we introduce an `operator-valued distribution'
$X\ni x\mapsto\partial_x^\dag$ so that, for any $h\in C_0(X)$ and $f^{(n)}\in \mathbf B_0^{Q}(X^n)$,
\begin{equation}\label{hbioghogy}
\int_X dx\, h(x)\partial_x^\dag f^{(n)}:
=h \cd f^{(n)}.
\end{equation}
 In other words, we may think  $\partial_x^\dag f^{(n)}=\delta_x\cd f^{(n)}$. Thus,
\begin{equation}\label{ytdee67y}
 \mathbf J^+(h)=\int_X dx\, h(x)\partial_x^\dag\,.
\end{equation}

For $h\in C_0(X)$, we will now need operators
$$\int_X dx\, h(x) \partial_x^\dag \partial_x,\qquad \int_X dx\, h(x)\partial_x^\dag \partial_x\partial_x$$
acting on $\mathbf F^{Q}_{\mathrm{fin}}(B_0(X))$.
In view of \eqref{kbhfdvyd} and \eqref{hbioghogy}, we have, for each $f^{(n)}\in \mathbf B_0^{Q}(X^n)$,
\begin{align*}
 \bigg(\int_X dx\, h(x) \partial_x^\dag \partial_x\,f^{(n)}\bigg)(x_1,\dots,x_n)&=n\bigg(\int_X dx\, h(x)\di_x^\dag f^{(n)}(x,\cdot)\bigg)(x_1,\dots,x_n)\\
 &=n\,\Sym_n\big(h(x_1)f^{(n)}(x_1,x_2,\dots,x_n)\big)
 \end{align*}
(compare with \eqref{jfiu}), and
\begin{align}
 &\bigg(\int_X dx\, h(x) \partial_x^\dag \partial_x\di_x\,f^{(n)}\bigg)(x_1,\dots,x_{n-1})\notag\\
 &\quad =n(n-1)\bigg(\int_X dx\, h(x)\di_x^\dag f^{(n)}(x,x,\cdot)\bigg)(x_1,\dots,x_{n-1})\notag\\
 &\quad =n(n-1)\,\Sym_{n-1}\big(h(x_1)f^{(n)}(x_1,x_1,x_2,x_3,\dots,x_{n-1})\big).\label{igiugygg}
 \end{align}

\begin{theorem}\label{hfu8fr78} Assume \eqref{vggyufd7u} holds. Then for $h\in C_0(X)$,
\begin{align}
\mathbf J^0(h)&=\int_X dx\, h(x)\lambda \partial_x^\dag \partial_x,\label{vfyter}\\
\mathbf J^-_1(h)&=\int_X dx\, h(x)\eta \partial_x^\dag \partial_x\partial_x\,.\label{bkvgutfgi}
\end{align}
Thus,
\begin{equation}\label{ydfyd}\mathbf J(h)=\int_X dx\, h(x)(\partial_x^\dag+\lambda\partial_x^\dag \partial_x +\eta \partial_x^\dag \partial_x\partial_x+\partial_x).\end{equation}
\end{theorem}

In view of formula \eqref{utr7r}, the operator $\mathbf J(h)$ is a realization of the operator $\la \omega,h\ra$ in the extended anyon Fock space $\mathbf F^{Q}(\mathcal H,\nu)$. So, with an abuse of notation, we may denote  $\mathbf J(h)$ by 
$\la \omega,h\ra$.
Then, by \eqref{yf7r} and \eqref{ydfyd}, we get the following representation of the anyon Meixner-type white noise (realized in the extended anyon Fock space $\mathbf F^{Q}(\mathcal H,\nu)$): 
$$ \omega(x)=\partial_x^\dag+\lambda\partial_x^\dag \partial_x +\eta \partial_x^\dag \partial_x\partial_x+\partial_x.$$

\begin{remark} We note that, for $q$-commutation relations with $q$ being real from either the interval $(-1,0)$ or the interval $(0,1)$ \cite{BS,BKS,A1}, there is no analog of a $q$-L\'evy process which would have a representation like in \eqref{ydfyd}. Nevertheless, as shown in \cite{BW}, there exist classical Markov processes whose transition probabilities are measures of orthogonality for $q$-Meixner (orthogonal) polynomials on the real line.
\end{remark}

\section{Proofs}\label{ur867}

\subsection{Proof of Proposition}\ref{yfd6rd6}

Note that $\Sym_1=\mathbf 1$, so we need to prove the statement for $n\ge2$. Following \cite{BLW}, let us first briefly recall how one shows that the operator $\Sym_n$ given by
\eqref{oitr8o} is an orthogonal projection in the space $L^2(X^n,m^{\otimes n})$.
For each $\pi\in \mathfrak S_n$, we define
\begin{equation}\label{hvtd7i5} (\Psi_\pi f^{(n)})(x_1,\dots,x_n)=Q_\pi(x_1,\dots,x_n)f^{(n)}(x_{\pi^{-1}(1)},\dots,x_{\pi^{-1}(n)})\end{equation} for $(x_1,\dots,x_n)\in X^{(n)}$. Thus,
$ \Sym_n=\frac1{n!}\sum_{\pi\in\mathfrak  S_n}\Psi_{\pi}$.
We then have
$\Psi_\pi^*=\Psi_{\pi^{-1}}$,
which implies $\Sym_n^*=\Sym_n$. Furthermore, for each permutation $\varkappa\in \mathfrak S_n$, we have
\begin{equation}\label{giyr86ort}
\Psi_\pi\Psi_\varkappa=\Psi_{\varkappa\pi}.
\end{equation}
Therefore, on $X^{(n)}$,
\begin{equation}\label{ft7e6ws} \Sym_n^2=\frac1{(n!)^2}\sum_{\pi,\varkappa\in \mathfrak S_n}\Psi_\pi\Psi_\varkappa =\frac1{(n!)^2}\sum_{\pi\in\mathfrak S_n}
\sum_{\varkappa\in \mathfrak S_n}\Psi_{\pi\varkappa}
=\frac1{n!}\sum_{\varkappa\in \mathfrak S_n}\Psi_{\varkappa}=\Sym_n.
\end{equation}
Thus, $\Sym_n$ is an orthogonal projection.
 Note that formula \eqref{giyr86ort} implies that, for $\varkappa,\pi\in\mathfrak S_n$,
\begin{multline}\label{jgi86rt} \left(
\prod_{\substack{1\le i<j\le n\\ \pi(i)>\pi(j)}}\!\!\! Q(x_i,x_j)
\right)
\left(
\prod_{\substack{1\le k<l\le n\\ \varkappa(k)>\varkappa(l)}}\!\!\!Q(x_{\pi^{-1}(k)},x_{\pi^{-1}(l)})
\right)\\
=
\prod_{\substack{1\le i<j\le n\\ (\varkappa\pi)(i)> (\varkappa\pi)(j)}} \!\!\!\!\!\!Q(x_i,x_j),\quad (x_1,\dots,x_n)\in X^{(n)}.
\end{multline}

Now let us  consider the linear, bounded operator $\Sym_n$ in $L^2(X^n,m_\nu^{(n)})$. We represent the operator $\Sym_n$ as
 \begin{equation}\label{gygfy}\Sym_n=\frac1{n!}\sum_{\pi\in \mathfrak S_n}\Psi_{\pi},\end{equation}
  with $ \Psi_\pi f^{(n)}$ being defined on the whole $X^n$ by the right hand side of  formula \eqref{hvtd7i5} in which the function $Q_\pi(x_1,\dots,x_n)$ on $X^n$ is 
  defined in subsec.~\ref{guyf7r}.

We fix a permutation $\pi\in \mathfrak S_n$
and a partition $\theta=\{\theta_1,\dots,\theta_l\}\in\Pi(n)$ satisfying \eqref{ft6e}, and let \eqref{adiov}, \eqref{huyfr7i5er} hold. Further, let $\varkappa\in\mathfrak  S_n$ and let $\zeta=\{\zeta_1,\dots,\zeta_l\}\in\Pi(n)$ be such that
\begin{equation}\label{iut8o6t8}\max\zeta_1<\max\zeta_2<\dots<\max\zeta_l,\end{equation}
and
$$\varkappa \beta_i=\zeta_{\widehat\varkappa(i)},\quad i=1,\dots,l,$$
where $\widehat\varkappa\in \mathfrak S_l$.

Then, for each function $f^{(n)}:X^n\to\mathbb C$ and $(x_1,\dots,x_n)\in X^n$,
we have
\begin{equation}\label{uyr5e67e} (\Psi_\pi  \Psi_\varkappa f^{(n)})(x_1,\dots,x_n)=\left(
\prod_{\substack{1\le i<j\le l\\ \widehat \pi(i)>\widehat \pi(j)}}\!\!\! Q(x_{\theta_i},x_{\theta_j})
\right)( \Psi_\varkappa f^{(n)})(x_{\pi^{-1}(1)},\dots,x_{\pi^{-1}(n)}).\end{equation}
Denote $y_i=x_{\pi^{-1}(i)}$, or equivalently $y_{\pi(i)}=x_i$ for $i=1,\dots,n$.  Thus, $y_{\pi(i)}=y_{\pi(j)}$ if and only if $i$ and $j$ belong to the same element of the partition $\theta$. Equivalently, $y_i=y_j$ if and only if $i$ and $j$ belong to the same element of the partition $\beta$. Therefore,
$$ (\Psi_\varkappa f^{(n)})(y_1,\dots,y_n)=\left(
\prod_{\substack{1\le u<v\le l\\ \widehat \varkappa(u)>\widehat \varkappa(v)}}\!\!\! Q(y_{\beta_u},y_{\beta_v})
\right)f^{(n)}(y_{\varkappa^{-1}(1)},\dots,y_{\varkappa^{-1}(n)}). $$
Hence,
\begin{align} &( \Psi_\varkappa f^{(n)})(x_{\pi^{-1}(1)},\dots,x_{\pi^{-1}(n)})\notag\\
&\quad =\left(
\prod_{\substack{1\le u<v\le l\\ \widehat \varkappa(u)>\widehat \varkappa(v)}}\!\!\! Q(x_{\pi^{-1}\beta_u},x_{\pi^{-1}\beta_v})
\right)f^{(n)}(x_{\pi^{-1}\varkappa^{-1}(1)},\dots,x_{\pi^{-1}\varkappa^{-1}(n)})\notag\\
&\quad = \left(
\prod_{\substack{1\le u<v\le l\\ \widehat \varkappa(u)>\widehat \varkappa(v)}}\!\!\! Q(x_{
\theta_{\widehat \pi^{-1}(u)}
},x_{
\theta_{\widehat \pi^{-1}(v)}
})
\right)f^{(n)}(x_{(\varkappa\pi)^{-1}(1)},\dots,x_{(\varkappa\pi)^{-1}(n)}),\label{gyu7r754}
\end{align}
where we used the observation that, for each $u=1,\dots,l$,
$$ \pi^{-1}\beta_u=\theta_{\widehat\pi^{-1}(u)}.$$
Using \eqref{jgi86rt}, we get
\begin{equation}\label{uiytr8o6}
\left(
\prod_{\substack{1\le i<j\le l\\ \widehat \pi(i)>\widehat \pi(j)}}\!\!\! Q(x_{\theta_i},x_{\theta_j})
\right)\left(
\prod_{\substack{1\le u<v\le l\\ \widehat \varkappa(u)>\widehat \varkappa(v)}}\!\!\! Q(x_{
\theta_{\widehat \pi^{-1}(u)}
},x_{
\theta_{\widehat \pi^{-1}(v)}
})
\right)= \left(
\prod_{\substack{1\le i<j\le l\\ \widehat{\varkappa\pi}(i)>\widehat {\varkappa\pi}(j)}}\!\!\! Q(x_{\theta_i},x_{\theta_j})
\right).
\end{equation} Here $\widehat{\varkappa\pi}$ is the permutation from $\mathfrak S_l$ induced by the permutation $\varkappa\pi\in \mathfrak S_n$ and the partition $\theta$. In \eqref{uiytr8o6}, we used the observation that
$\widehat{\varkappa\pi}=\widehat\varkappa\,\widehat\pi$.
Now, substituting \eqref{gyu7r754} into \eqref{uyr5e67e}
and using \eqref{uiytr8o6}, we conclude that
\begin{equation}\label{gciugfi} \Psi_\pi \Psi_\varkappa= \Psi_{\varkappa\pi},\end{equation}  and hence analogously to \eqref{ft7e6ws}, we get $\Sym_n^2=\Sym_n$.

Next, we note that  the measure $m^{(n)}_\nu$
remains invariant under the transformation
$$ X^n\ni (x_1,\dots,x_n)\mapsto (x_{\pi^{-1}(1)},\dots x_{\pi^{-1}(n)})\in X^n.$$ Furthermore, as easily seen, the equality
$$ \overline{Q_{\pi^{-1}}(x_{\pi^{-1}(1)},\dots,x_{\pi^{-1}(n)})}= Q_\pi(x_{1},\dots,x_{n})$$
holds for each  $(x_1,\dots,x_n)\in X^n$. Hence, for each $\pi\in \mathfrak S_n$, $ \Psi_\pi^*= \Psi_{\pi^{-1}}$, which implies $\Sym_n^*=\Sym_n$.

Thus, $\Sym_n$ is an orthogonal projection in $L^2(X^n,m^{(n)}_\nu)$. Analogously to \cite[Proposition 2.5]{BLW}, we easily conclude that the image of $\Sym_n$is indeed $\mathbf F^Q_n(\mathcal H,\nu)$. Thus, Proposition~\ref{yfd6rd6} is proven.

Recall the tensor product $\cd$ defined on 
 $\mathbf F_{\mathrm{fin}}^{Q}(B_0(X))$ by formula \eqref{vytfr7i6ed}. Using \eqref{gygfy} and \eqref{gciugfi}, it is easy to show that, for any $f^{(m)}\in B_0(X^m)$ and $g^{(n)}\in B_0(X^n)$, we have
\begin{align*}(\Sym_m\, f^{(m)})\cd (\Sym_n\, g^{(n)})&=\Sym_{m+n}( (\Sym_m \,f^{(m)})\otimes (\Sym_n\, g^{(n)}))\\
&=\Sym_{m+n}(f^{(m)}\otimes g^{(n)}).\end{align*}
Therefore, the tensor product  $\cd$ is associative on  $\mathbf F_{\mathrm{fin}}^{Q}(B_0(X))$.

\subsection{Proof of Theorem \ref{utu8}}
Recall the unitary operator $I:L^2(\tau)\to\mathcal F^Q(\mathcal G)$. Our next aim is to obtain an explicit form of the subspace $I(\mathscr{OP}_n)$ of $\mathcal F^Q(\mathcal G)$.

Denote by $\mathbb N_{0,\,\mathrm{fin}}^\infty$ the set of all infinite sequences $\alpha=(\alpha_0,\alpha_1,\alpha_2,\dots)\in\mathbb N_0^\infty$ such that only a finite number of $\alpha_j$'s are not equal to zero. Let $|\alpha|:=\alpha_0+\alpha_1+\alpha_2+\dotsm$. For each $\alpha\in \mathbb N_{0,\,\mathrm{fin}}^\infty$ with $|\alpha|\ge1$, we denote by $\mathcal F_\alpha$ the subspace of the Fock space $\mathcal F^Q(\mathcal G)$ which consists of all elements of the form
$$ \Sym_{|\alpha|}\big(f^{(|\alpha|)}(x_1,\dots,x_{|\alpha|})p_{0}(s_1)\dotsm p_{0}(s_{\alpha_0})p_{1}(s_{\alpha_0+1})
\dotsm p_{1}(s_{\alpha_0+\alpha_1})p_{2}(s_{\alpha_0+\alpha_1+1})\dotsm\big),$$
where $f^{(|\alpha|)}\in\mathcal H^{\otimes|\alpha|}$. For $\alpha\in \mathbb N_{0,\,\mathrm{fin}}^\infty$ with $|\alpha|=0$, we set $\mathcal F_\alpha:=\{c\Omega\mid c\in\mathbb C\}$.
The following proposition is proven in \cite[Section 7]{BLW}.
This result is a counterpart of the  Nualart--Schoutens decomposition of the $L^2$-space of a classical L\'evy process \cite{NS}, see also \cite{Schoutens}.

\begin{proposition}\label{ui8t868} We have
\begin{equation}\label{vytd}\mathcal F^Q(\mathcal G)=\bigoplus_{\alpha\in \mathbb N_{0,\,\mathrm{fin}}^\infty}\mathcal F_\alpha.\end{equation}
\end{proposition}

For each $n\in\mathbb N_0$, we define
$$\mathbb F_{n}:=\bigoplus_{\substack{\alpha\in \mathbb N_{0,\,\mathrm{fin}}^\infty\\
\alpha_0+2\alpha_1+3\alpha_2+\dotsm=n}} \mathcal F_\alpha.$$
Note that, by \eqref{vytd}, $$\mathcal F^Q(\mathcal G)=\bigoplus_{n=0}^\infty \mathbb F_n.$$

\begin{proposition}\label{hjvfytdc}
For each $n\in\mathbb Z_+$,
$$I \mathscr{OP}_n=\mathbb F_n.$$
\end{proposition}

\begin{proof}
It suffices to prove that, for each $n\in\mathbb N$,
\begin{equation}\label{uf7yrd}
I \mathscr{MP}_n=\bigoplus_{\substack{\alpha\in \mathbb N_{0,\,\mathrm{fin}}^\infty\\
\alpha_0+2\alpha_1+3\alpha_2+\dotsm\le n}} \mathcal F_\alpha=:\mathbb M_n.
\end{equation}

\begin{lemma}\label{tye6ue}
The space $\mathbb M_n$ consists of all  finite sums of elements of the form
\begin{equation}\label{iu87re65} \Sym_k\big(f^{(k)}(x_1,\dots,x_k)s_1^{i_1}s_2^{i_2}\dotsm s_k^{i_k}\big),\end{equation}
where $f^{(k)}\in\h^{\otimes k}$ and $i_1+i_2+\dots+i_k+k\le n$.
\end{lemma}

\begin{proof} For each $\pi\in\mathfrak S_k$, we define a unitary operator $\Psi_\pi$ on $(\mathcal H\otimes L^2(\mathbb R,\nu))^{\otimes k}$ by
$$(\Psi_\pi g^{(k)})(x_1,s_1,\dots,x_k,s_k)=Q_\pi(x_1,\dots,x_k)g^{(k)}(x_{\pi^{-1}(1)},s_{\pi^{-1}(1)},\dots,x_{\pi^{-1}(k)},s_{\pi^{-1}(k)}).$$
Here the function $Q_\pi$ is defined   by \eqref{y7e57ie}. Then, by \cite{BLW}, the operators $\Psi_\pi$ form a unitary representation of the symmetric group $\mathfrak S_k$, and for each $\pi\in \mathfrak S_k$ we have $\Sym_k=\Sym_k\Psi_\pi$. Hence, for any permutation $\pi\in \mathfrak S_k$,   $u^{(k)}\in\mathcal H^{\otimes k}$, and any polynomial $r^{(k)}(s_1,\dots,s_k)$ in the $s_1,\dots,s_k$ variables,
$$ \Sym_k\big(
f^{(k)}(x_1,\dots,x_k)r^{(k)}(s_1,\dots,s_k)
\big)=\Sym_k\big(
u^{(k)}(x_1,\dots,x_k) r^{(k)}(s_{\pi^{-1}(1)},\dots,s_{\pi^{-1}(k)})
\big),$$
where
$$ u^{(k)}(x_1,\dots,x_k)=Q_\pi(x_1,\dots,x_k)f^{(k)}(x_{\pi^{-1}(1)},\dots,x_{\pi^{-1}(k)}).$$
 In particular, $u^{(k)}\in\mathcal H^{\otimes k}$.

Noting the evident representations
$$ p_l(s)=\sum_{i=0}^l \alpha_{il}\, s^i,\quad s^l=\sum_{i=0}^l
\beta_{il}\,p_i(s),$$
we easily conclude the  lemma.
\end{proof}

We  now  finish the proof of \eqref{uf7yrd}.  Let $\mathcal F_{\mathrm{fin}}(\mathcal H\otimes\mathscr P(\mathbb R))$ be the linear subspace of the full Fock space over $\mathcal H\otimes L^2(\mathbb R,\nu)$ which consists of finite sums of $c\Omega$ ($c\in\mathbb C$) and elements of the form
\begin{equation}\label{fgufr} f^{(k)}(x_1,\dots,x_k)s_1^{i_1}s_2^{i_2}\dotsm s_k^{i_k}\end{equation}
with $f^{(k)}\in\h^{\otimes k}$, $i_1,i_2,\dots,i_k\in\mathbb Z_+$, $k\in\mathbb N$.  We set
\begin{equation}\label{iuagiuai} \Sym:=\mathbf 1\oplus \Sym_1\oplus \Sym_2\oplus \Sym_3\oplus\dotsm\,.\end{equation}
 This operator  projects $\mathcal F_{\mathrm{fin}}(\mathcal H\otimes\mathscr P(\mathbb R))$ onto $\mathcal F_{\mathrm{fin}}^Q(\h\otimes\mathscr P(\mathbb R))$.
We have,
for each $h\in C_0(X)$ and $F\in \mathcal F_{\mathrm{fin}}(\mathcal H\otimes\mathscr P(\mathbb R))$,
\begin{gather}
a^+(h\otimes 1) \Sym F= \Sym J^+(h\otimes 1)F,\quad a^-(h\otimes 1) \Sym F= \Sym J^-(h\otimes 1)F,\label{ihoy9y6}\\
a^0(h\otimes \id) \Sym F= \Sym J^0(h\otimes \id)F.\label{gfur7}
\end{gather}
Here, for each $F$ as in \eqref{fgufr},
\begin{align}
&(J^+(h\otimes 1)F)(x_1,s_1,\dots,x_{k+1},s_{k+1})=h(x_1)1(s_1)f^{(k)}(x_2,\dots,x_{k+1})s_2^{i_1}s_3^{i_2}\dotsm s_{k+1}^{i_k},\notag\\
&(J^0(h\otimes \id)F)(x_1,s_1,\dots,x_{k},s_{k})=(h(x_1)s_1+\dots+h(x_k)s_k)f^{(k)}(x_1,\dots,x_k)s_1^{i_1}s_2^{i_2}\dotsm s_k^{i_k},\notag\\
&(J^-(h\otimes 1)F)(x_1,s_1,\dots,x_{k-1},s_{k-1})=\sum_{j=1}^k \int_X dy\int_{\mathbb R}\nu(dt)\, h(y)
Q(y,x_1)\dotsm Q(y,x_{j-1})\notag\\
&\qquad\times f^{(k)}(x_1,\dots,x_{j-1},y,x_{j},\dots,x_{k-1})s_1^{i_1}\dotsm s_{j-1}^{i_{j-1}}t^{i_j}s_j^{i_{j+1}}\dotsm s_{k-1}^{i_{j_k}}.\label{vtydy7}
\end{align}
Hence, it follows by induction from Lemma~\ref{tye6ue} and \eqref{ihoy9y6}--\eqref{vtydy7} that
$$\la\omega, h_1\ra\dotsm\la\omega, h_n\ra\Omega\subset \mathbb M_n$$
for  any $h_1,\dots,h_n\in C_0(X)$, $n\in\mathbb N$. Since $\mathbb M_n$ is a closed subspace of $\mathcal F^Q(\mathcal G)$, we therefore get the inclusion $I\mathscr{MP}_n\subset \mathbb M_n$. On the other hand, it directly follows from the proof of \cite[Proposition 6.7]{BLW} that each element of $\mathbb M_n$ which has form \eqref{iu87re65} belongs to $I\mathscr{MP}_n$. Hence, we get the inverse inclusion $\mathbb M_n\subset I\mathscr{MP}_n$.
\end{proof}

Note that, for each $h\in C_0(X)$,
\begin{equation}\label{kjguyfr7u} a^0(h\otimes \id)=d\Gamma(M_{h\otimes \id})=d\Gamma(M_h\otimes M_{\id}),\end{equation}
where $M_h$ is the operator of multiplication by the function $h(x)$ in $\mathcal H$ and $M_{\id}$ is the (restriction to $\mathscr P(\mathbb R)$ of the) operator of multiplication by the monomial $\id(s)=s$ in $L^2(\R,\nu)$.
(Note that the operator $M_{\id}$ is unbounded in $L^2(\mathbb R,\nu)$ if the support of measure $\nu$ is unbounded, and the second quantization operator has domain $\mathcal F_{\mathrm{fin}}^Q(\h\otimes\mathscr P(\mathbb R))$.)
In view of the recursion  formula \eqref{hdtrss},
we get the representation
$$ M_{\id}=A^++A^0+A^-,$$
where $A^+$, $A^0$, and $A^-$ are the linear operators on $\mathscr P(\mathbb R)$ given by
\begin{equation}\label{nitg} A^+p_k:=p_{k+1},\quad A^0p_k=b_kp_k,\quad A^- p_k=a_k p_{k-1}.\end{equation}
By \eqref{kjguyfr7u} and \eqref{nitg},
\begin{equation}\label{uit8rt} a^0(h\otimes \id)= d\Gamma(M_h\otimes A^+)+d\Gamma(M_h\otimes A^0)+d\Gamma(M_h\otimes A^-).
\end{equation}

By \eqref{g7r75e} and \eqref{uit8rt}, we get, for each $h\in C_0(X)$,
\begin{equation}\label{igyur7} \la \omega,h\ra=\mathcal A^+(h)+\mathcal A^0(h) +\mathcal A^-(h),\end{equation}
where
\begin{align}
\mathcal A^+(h):&=a^+(h\otimes 1)+d\Gamma(M_h\otimes A^+),\notag\\
\mathcal A^0(h):&=d\Gamma(M_h\otimes A^0),\notag\\
\mathcal A^-(h):&=a^-(h\otimes 1)+d\Gamma(M_h\otimes A^-).\label{f75re75}
\end{align}

\begin{proposition}\label{giyuf7or6o}
For each $h\in C_0(X)$, we have $\mathcal A^+(h):\mathbb F_n\to\mathbb F_{n+1}$, $\mathcal A^0(h):\mathbb F_n\to \mathbb F_n$, $\mathcal A^-(h):\mathbb F_n\to \mathbb F_{n-1}$.
\end{proposition}

\begin{proof}

Let $\sharp=+,\,0,\,-$. For each $h\in C_0(X)$, we define an operator $N(M_h\otimes A^\sharp)$ on $\mathcal F_{\mathrm{fin}}(\mathcal H\otimes\mathscr P(\mathbb R))$ by setting
$N(M_h\otimes A^\sharp)\Omega:=0$ and for each $n\in\mathbb N$,
\begin{multline*}N(M_h\otimes A^\sharp)\restriction \big(\mathcal F_{\mathrm{fin}}(\mathcal H\otimes\mathscr P(\mathbb R))\cap \mathcal G^{\otimes n}\big)\\
:=
(M_h\otimes A^\sharp)\otimes \mathbf 1\otimes\dots\otimes \mathbf 1+
\mathbf 1\otimes (M_h\otimes A^\sharp)\otimes \mathbf 1\otimes\dots\otimes \mathbf 1+\dots+\mathbf 1\otimes\dots\otimes \mathbf 1\otimes
(M_h\otimes A^\sharp).
\end{multline*}

\begin{lemma}\label{lgti8t}
 Let $\sharp=+,\,0,\,-$. For any  $h\in C_0(X\mapsto\mathbb R)$
and $F\in \mathcal F_{\mathrm{fin}}(\mathcal H\otimes\mathscr P(\mathbb R))$, we have
$$d\Gamma(M_h\otimes A^\sharp) \Sym F= \Sym N(M_h\otimes A^\sharp)F.$$
\end{lemma}

\begin{proof} Fix any  $F\in \mathcal F_{\mathrm{fin}}(\mathcal H\otimes\mathscr P(\mathbb R))$  of the form
$$ F(x_1,s_1,\dots,x_n,s_n)=f^{(n)}(x_1,\dots,x_n)p_{i_1}(s_1)\dotsm p_{i_n}(s_n).$$ By \eqref{oitr8o},
\begin{align}
&(\Sym_nF)(x_1,s_1,\dots,x_k,s_k)\notag\\
&\quad
=\frac1{n!}\sum_{\pi\in \mathfrak S_n}Q_\pi(x_1,\dots,x_n)f^{(n)}(x_{\pi^{-1}(1)},\dots,x_{\pi^{-1}(n)})p_{i_1}
(s_{\pi(1)})\dotsm p_{i_n}(s_{\pi(n)}).\label{gyfu7e5}
\end{align}
Note that
\begin{equation}\label{vuyd7ik}
d\Gamma(M_h\otimes A^+)= \Sym N(M_h\otimes A^+).
\end{equation}
 By \eqref{gyfu7e5},
\begin{align}
&(N(M_h\otimes A^+)\Sym_nF)(x_1,s_1,\dots,x_n,s_n)\notag\\
&\quad
=\frac1{n!}\sum_{j=1}^n \sum_{\pi\in \mathfrak S_n}Q_\pi(x_1,\dots,x_n)
h(x_{\pi^{-1}(j)})
f^{(n)}(x_{\pi^{-1}(1)},\dots,x_{\pi^{-1}(n)})\notag\\
&\qquad\times p_{i_1}
(s_{\pi^{-1}(1)})\dotsm p_{i_j+1}(s_{\pi^{-1}(j)})\dotsm p_{i_n}(s_{\pi^{-1}(n)})\notag\\
&\quad
=\frac1{n!}\sum_{j=1}^n \sum_{\pi\in \mathfrak S_n}Q_\pi(x_1,\dots,x_n)g_j^{(n)}(x_{\pi^{-1}(1)},s_{\pi^{-1}(1)},\dots,x_{\pi^{-1}(n)},s_{\pi^{-1}(n)}).
\label{hgit8p}
\end{align}
Here, for $j=1,\dots,n$,
$$ g_j^{(n)}(x_1,s_1,\dots,x_n,s_n):=h(x_j)f^{(n)}(x_1,\dots,x_n)p_{i_1}(s_1)\dotsm p_{i_j+1}(s_j)\dotsm p_{i_n}(s_n).$$
Then, by \eqref{vuyd7ik} and \eqref{hgit8p},
\begin{align}
&(d\Gamma(M_h\otimes A^+)\Sym_nF)(x_1,s_1,\dots,x_n,s_n)\notag\\
&\quad=
\frac1{(n!)^2}\sum_{j=1}^n \sum_{\sigma\in \mathfrak S_n}\sum_{\pi\in \mathfrak S_n}
Q_\sigma(x_1,\dots,x_n)Q_\pi(x_{\sigma^{-1}(1)},\dots,x_{\sigma^{-1}(n)}) \notag\\
&\qquad \times
g^{(n)}_j(x_{\sigma^{-1}(\pi^{-1}(1))},s_{\sigma^{-1}(\pi^{-1}(1))},\dots x_{\sigma^{-1}(\pi^{-1}(n))},s_{\sigma^{-1}(\pi^{-1}(n))}).
\notag
\end{align}
Hence,
\begin{align*}d\Gamma(M_h\otimes A^+)\Sym_nF&= \sum_{j=1}^n \Sym_n^2g_j^{(n)}=
\sum_{j=1}^n \Sym_ng_j^{(n)}
\\
&=\Sym_n\left(\sum_{j=1}^n g_j^{(n)}\right)= \Sym_n N(M_h\otimes A^+)F.\end{align*}
The proof for $A^0$ and $A^-$ is analogous.
\end{proof}

Now, the proposition follows directly from the definition of the spaces $\mathbb F^{(n)}$, formula \eqref{ihoy9y6}, and Lemma \ref{lgti8t}.
\end{proof}

\begin{proposition}\label{tyfd6ure}
For any $h_1,\dots,h_n\in C_0(X)$, we have
$$ I \la P_n(\omega),h_1\otimes\dots \otimes h_n\ra=\mathcal A^+(h_1)\dotsm\mathcal A^+(h_n)\Omega.$$
\end{proposition}

\begin{proof} Recall that $\la P_n(\omega),h_1\otimes\dots \otimes h_n\ra$ is the orthogonal projection of the monomial
$$\la h_1,\omega\ra\dotsm\la h_n,\omega\ra=\la h_1\otimes\dots\otimes h_n,\omega^{\otimes n}\ra$$ onto
$\mathscr {OP}_n$. The statement follows from Propositions \ref{hjvfytdc} and \ref{giyuf7or6o} if we note that $$I \la P_n(\omega),h_1\otimes\dots \otimes h_n\ra$$
is equal to the orthogonal projection of
$$\la \omega,h_1\ra\dotsm\la \omega,h_n\ra\Omega=
(\mathcal A^+(h_1)+\mathcal A^0(h_1)+\mathcal A^-(h_1))\dotsm
(\mathcal A^+(h_n)+\mathcal A^0(h_n)+\mathcal A^-(h_n))\Omega$$
onto $\mathbb F_n$.
\end{proof}

We will now explicitly calculate the vector $I \la P_n(\omega),h_1\otimes\dots \otimes h_n\ra$. We introduce a topology on  $B_0(X^n)$ which yields the following notion of convergence: $f_n\to f$ as $n\to\infty$ means that there exists a set $\Delta\in\mathcal B_0(X)$ such that
$\operatorname{supp}(f_n)\subset\Delta$ for all $n\in\mathbb N$ and
\eqref{ilyufre75iei7} holds.
Note that $C_0(X^n)$ is a topological subspace
of $B_0(X^n)$. 

For each $\theta=\{\theta_1,\dots,\theta_l\}\in\Pi(n)$ with $\theta_1,\dots,\theta_l$ satisfying 
\eqref{ft6e}, we define, for $f^{(n)}\in B_0(X^n)$, $(x_1,\dots,x_l)\in X^{(l)}$, and $(s_1,\dots,s_l)\in\R^l$,
\begin{equation}\label{isgdiy}(\mathcal E_\theta  f^{(n)})(x_1,s_1,\dots,x_l,s_l):=f^{(n)}_\theta
(x_1,\dots,x_l) p_{|\theta_1|-1}(s_1)p_{|\theta_2|-1}(s_2)\dotsm p_{|\theta_l|-1}(s_l).\end{equation}
Here   the function $f^{(n)}_\theta
(x_1,\dots,x_l)$ is obtained from the function $f^{(n)}(y_1,\dots,y_n)$ by replacing $y_{i_1}$ with $x_1$ for all $i_1\in\theta_1$, $y_{i_2}$ with $x_2$ for all $i_2\in\theta_2$,  and so on. Note that the function $f^{(n)}_\theta: X^{(l)}\to\mathbb C$ is completely identified by the restriction of the function $f^{(n)}:X^n\to\mathbb C$ to the set $X_\theta^{(n)}$.

For example, let $n=6$ and let $\theta=\{\theta_1,\theta_2,\theta_3\}\in\Pi(6)$ be of the form
$$\theta_1=\{1,3\},\quad\theta_2=\{2,4,6\},\quad\theta_3=\{5\}.$$
Then, for each $(x_1,x_2,x_3)\in X^{(3)}$ and $(s_1,s_2,s_3)\in\R^3$,
$$(\mathcal E_\theta f^{(6)})(x_1,s_1,x_2,s_2,x_3,s_3)
=f^{(6)}(x_1,x_2,x_1,x_2,x_3,x_2)p_1(s_1)p_2(s_2)p_0(s_3).$$

\begin{proposition}\label{iyr867r} For each $n\in\mathbb N$, the mapping
$$C_0(X)^n\ni(h_1,\dots,h_n)\mapsto \la P_n(\omega),h_1\otimes\dots \otimes h_n\ra\in L^2(\tau)$$
may be extended by linearity and continuity to a mapping
$$ B_0(X^n)\ni f^{(n)}\to \la P_n(\omega),f^{(n)}\ra\in L^2(\tau).$$ Furthermore,
for each $f^{(n)}\in B_0(X^n)$, we have
\begin{equation}\label{ouifr7ored} I\la P_n(\omega),f^{(n)}\ra=\Sym\left(\sum_{\theta\in \Pi(n)}
\mathcal E_\theta f^{(n)}\right).\end{equation}
\end{proposition}

\begin{proof}
Fix any $h_1,\dots,h_n\in C_0(X)$ and set $f^{(n)}(x_1,\dots,x_n)=h_1(x_1)\dotsm h_n(x_n)$. Then, by Proposition~\ref{tyfd6ure}, formula \eqref{ouifr7ored} is equivalent to
\begin{equation}
\label{ufdr6s}
\big(a^+(h_1\otimes 1)+d\Gamma(M_{h_1}\otimes A^+)\big)\dotsm
\big(a^+(h_n\otimes 1)+d\Gamma(M_{h_n}\otimes A^+)\big)\Omega
  =\Sym\left(\sum_{\theta\in \Pi(n)}
\mathcal E_\theta f^{(n)}\right).
\end{equation}
By \eqref{ihoy9y6} and Lemmma~\ref{lgti8t}, formula \eqref{ufdr6s} would follow from
\begin{equation}\label{igf7r}
\big(J^+(h_1\otimes 1)+N(M_{h_1}\otimes A^+)\big)
\dotsm \big(J^+(h_n\otimes 1)+N(M_{h_n}\otimes A^+)\big)\Omega=\sum_{\theta\in \Pi(n)}
\mathcal E_\theta f^{(n)}.
\end{equation}

Let $\beta=\{\beta_1,\dots,\beta_k\}$ be an (unordered) partition of $\{i+1,i+2,\dots,n\}$. Then
\begin{equation}\label{fdtrs5} J^+(h_i\otimes 1)\mathcal E_\beta(h_{i+1}\otimes h_{i+2}\otimes\dots\otimes  h_n)=\mathcal E_{\beta^+}(h_i\otimes h_{i+1}\otimes \dots\otimes h_n),\end{equation}
where $\beta^+:=\{\{i\},\beta_1,\dots,\beta_k\}$ is a partition of $\{i,i+1,\dots,n\}$. Furthermore,
\begin{equation}\label{iufy7}N(M_{h_i}\otimes A^+)\mathcal E_\beta(h_{i+1}\otimes h_{i+2}\otimes\dots\otimes  h_n)=\sum_{j=1}^k \mathcal E_{\beta_j^0}(h_i\otimes h_{i+1}\otimes \dots\otimes h_n),\end{equation}
where $\beta_j^0$ is  the partition of $\{i,i+1,\dots,n\}$
obtained from $\beta$ by adding $i$ to the set $\beta_j$, i.e.,
$$ \beta_j^0:=\{\beta_1,\dots,\beta_j\cup \{i\},\dots \beta_k\}.$$
By \eqref{fdtrs5} and \eqref{iufy7}, formula \eqref{igf7r}  follows by induction.  Finally, the extension
of formula \eqref{ouifr7ored} to the case of a general
$f^{(n)}\in B_0(X^n)$ follows by linearity and approximation.
\end{proof}

We will now prove Theorem~\ref{utu8}. Even, a bit more generally, we will prove that formula \eqref{gilyr7e5if} holds for any $f^{(n)},g^{(n)}\in B_0(X^n)$.

We first note that it suffices to prove formula \eqref{gilyr7e5if} in the case where $f^{(n)}=g^{(n)}=h_1\otimes\dots\otimes h_n$ with $h_1,\dots,h_n\in B_0(X)$.
By Proposition \ref{iyr867r},
\begin{align}
&\big(\la P_n(\omega),f^{(n)}\ra , \,\la P_n(\omega), f^{(n)}\ra\big)_{L^2(\tau)}\notag\\
&\quad =\left(\sum_{\theta\in\Pi(n)}\Sym_{|\theta|}
(\mathcal E_{\theta}f^{(n)}),\sum_{\zeta\in\Pi(n)}\Sym_{|\zeta|}
(\mathcal E_{\zeta}f^{(n)})\right)_{\mathcal F^Q(\mathcal G)}\notag\\
&\quad = \sum_{l=1}^n\sum_{\substack{\theta,\zeta\in\Pi(n)\\ |\theta|=|\zeta|=l}}\big(
\Sym_l(\mathcal E_\theta f^{(n)}),\mathcal E_\zeta f^{(n)})
\big)_{L^2((X\times \mathbb R)^l,(m\otimes \nu)^{\otimes l})}\,l!\,.\label{ouit8ot}
\end{align}
Note that, by Proposition \ref{yfd6rd6},
\begin{align}
(\Sym_n\,f^{(n)},\Sym_n\, f^{(n)})_{\mathbf F_{ n}^{Q}(\mathcal H,\nu)}&=
\int_{X^n}(\Sym_nf^{(n)})f^{(n)}\,dm^{(n)}_\nu\notag\\
&=\sum_{\zeta\in\Pi(n)} \int_{X_\zeta^{(n)}}(\Sym_n\,f^{(n)})f^{(n)}\,dm^{(n)}_{\nu,\, \zeta}\,.
\label{it89}\end{align}
By \eqref{ouit8ot} and \eqref{it89}, formula \eqref{gilyr7e5if} will follow if we show that, for a fixed $\zeta\in\Pi(n)$ with $|\zeta|=l$,
\begin{equation}\label{uit8p}
\sum_{\theta\in\Pi(n),\,|\theta|=l}
\big(
\Sym_l(\mathcal E_\theta f^{(n)}),\mathcal E_\zeta f^{(n)})
\big)_{L^2((X\times \mathbb R)^l,(m\otimes \nu)^{\otimes l})}\,l!= \int_{X_\zeta^{(n)}}(\Sym_n\,f^{(n)})f^{(n)}\,dm^{(n)}_{\nu,\, \zeta}\,.\end{equation}

So, let us fix a partition $\zeta=\{\zeta_1,\dots,\zeta_l\}\in\Pi(n)$ and assume that \eqref{iut8o6t8} holds.  Denote $k_i:=|\zeta_i|$, $i=1,\dots,l$.
We have, by the definition of $\mathcal E_\zeta f^{(n)}$:
\begin{equation}\label{rdu6e46i}(\mathcal E_\zeta f^{(n)})=\left(\prod_{i_1\in\zeta_1}
h_{i_1}\right)\otimes p_{k_1-1}\otimes\dots\otimes
\left(\prod_{i_l\in\zeta_l}
h_{i_l}\right)\otimes p_{k_l-1}.\end{equation}
 Let $\theta=\{\theta_1,\dots,\theta_l\}\in\Pi(n)$ and assume that \eqref{ft6e} holds. Let $r_i:=|\theta_i|$, $i=1,\dots,l$. We may assume that there exists a permutation $\widehat\pi\in\mathfrak S_l$ such that
 \begin{equation}\label{ydyduc}r_i=k_{\widehat\pi(i)},\quad i=1,\dots,l.\end{equation}
 Indeed, otherwise the corresponding term in the sum on the left hand side of formula \eqref{uit8p} vanishes.
Analogously to \eqref{rdu6e46i}, we have
\begin{align*}
&l!\, \Sym_l(\mathcal E_\theta f^{(n)})(y_1,s_1,\dots,y_l,s_l)=\sum_{\varkappa\in S_l}Q_\varkappa(y_1,\dots,y_l)\\
&\times
\left(\left(\prod_{j_1\in\theta_{\varkappa(1)}}h_{j_1}\right)\otimes p_{r_{\varkappa(1)}-1}\otimes\dots\otimes \left(\prod_{j_l\in\theta_{\varkappa(l)}}h_{j_l}\right)\otimes p_{r_{\varkappa(l)}-1}\right)(y_1,s_1,\dots,y_l,s_l).
\end{align*}
Hence, by \eqref{yufrur},
\begin{align}
&\big(
\Sym_l(\mathcal E_\theta f^{(n)}),\mathcal E_\zeta f^{(n)})
\big)_{L^2((X\times \mathbb R)^l,(m\otimes \nu)^{\otimes l})}\,l!\notag\\
&\quad =\sum_{\widehat\pi}\int_{X^l}Q_{\widehat\pi}(y_1,\dots,y_l)
\left(\prod_{j_1\in\theta_{\widehat\pi(1)}}h_{j_1}(y_1)\right)
\left(\prod_{i_1\in\zeta_1}
h_{i_1}(y_1)\right)\notag\\
&\qquad\times\dotsm\times
\left(\prod_{j_l\in\theta_{\widehat\pi(l)}}h_{j_l}(y_l)\right)
\left(\prod_{i_l\in\zeta_l}
h_{i_l}(y_l)\right)\,dy_1\dotsm  dy_l\, c_{k_1}\dotsm c_{k_l},\label{igr8678u}
\end{align}
where the summation is over all permutations $\widehat\pi\in S_l$ which satisfy \eqref{ydyduc}. Let us fix such a permutation $\widehat\pi$. Then, there exist
$$r_1! \dotsm r_l!=k_1!\dotsm k_l!$$ permutations $\pi\in \mathfrak S_n$ which satisfy
\begin{equation}\label{hit98t} \pi \zeta_i=\theta_{\widehat \pi(i)},\quad i=1,\dots,l.\end{equation}
Note that, for each permutation $\pi$  satisfying \eqref{hit98t} and for $(x_1,\dots,x_n)\in X_\zeta^{(n)}$,
\begin{align}&f^{(n)}(x_{\pi^{-1}(1)},\dots,x_{\pi^{-1}(n)})\notag\\
&\quad=(h_1\otimes \dots\otimes h_n)(x_{\pi^{-1}(1)},\dots,x_{\pi^{-1}(n)})\notag\\
&=(h_{\pi(1)}\otimes\dots\otimes h_{\pi(n)})(x_1,\dots,x_n)\notag\\
&\quad=\left(\prod_{j_1\in\pi\zeta_1}h_{j_1}\right)(y_1)\dotsm \left(\prod_{j_l\in\pi\zeta_l}h_{j_l}\right)(y_l)\notag\\
&\quad=\left(\prod_{j_1\in\theta_{\widehat\pi(1)}}h_{j_1}\right)(y_1)\dotsm \left(\prod_{j_l\in\theta_{\widehat\pi(l)}}h_{j_l}\right)(y_l),\label{fu88}
\end{align}
where $y_1=x_{i_1}$ for $i_1\in\zeta_1$,\dots, $y_l=x_{i_l}$ for $i_l\in\zeta_l$.

Let $\zeta, \theta \in \Pi(n)$ be such that condition \eqref{ydyduc} is satisfied by some permutation $\widehat{\pi} \in \mathfrak S_l$.
That is, the corresponding sequences $(k_1,\dots,k_l)$ and  $(r_1,\dots,r_l)$ coincide up to a permutation.
Denote
by $\mathfrak S_n[\zeta, \theta]$ the set of all permutations $\pi \in \mathfrak S_n$ which satisfy \eqref{hit98t} with some permutation $\widehat{\pi} \in \mathfrak S_l$.
(Note that the permutation $\widehat{\pi}$ is then completely identified by $\pi$, $\zeta$ and $\theta$ and automatically satisfies \eqref{hit98t}.)
Clearly, if $\theta$ and $\theta'$ are from $\Pi(n)$ with $|\theta| = |\theta'| = l$, both satisfying \eqref{hit98t}, and
 $\theta \ne \theta'$, then
 \begin{equation}\label{csd35}
\mathfrak S_n[\zeta, \theta] \cap \mathfrak S_n[\zeta,\theta'] = \varnothing.
\end{equation}
Furthermore,
\begin{align}\label{csd36}
\bigcup_{\substack{\theta \in\Pi(n),\, |\theta| =l \\ \theta\, \text{satisfying \eqref{hit98t}}}}
\mathfrak S_n[\zeta,\theta] = \mathfrak S_n.
\end{align}
Therefore,  by the definition of the measure $m_{c,\,\zeta}^{(n)}$ and formulas \eqref{uyr75rw}, \eqref{igr8678u}, \eqref{fu88}--\eqref{csd36},
\begin{align*}
&\big(\Sym_l(\mathcal E_\theta f^{(n)}),\mathcal E_\zeta f^{(n)}
\big)_{L^2((X\times \mathbb R)^l,(m\otimes\nu)^{\otimes l})}\,l!\\
& \quad = \frac{1}{n!}\sum_{\pi \in S_n[\zeta, \theta]}\int_{X_\zeta^{(n)}}\mathbf Q_\pi(x_1,\dots,x_n)f^{(n)}(x_{\pi^{-1}(1)}, \ldots , x_{\pi^{-1}(n)})\\
&\quad \quad \times f^{(n)}(x_1, \ldots, x_n)\,m_{\nu,\,\zeta}^{(n)}(dx_1\times\dots\times  dx_n).
\end{align*}Hence
\begin{align*}
&\sum_{\theta \in \Pi(n),\, |\theta| = l}\big(\Sym_l(\mathcal E_\theta f^{(n)}),\mathcal E_\zeta f^{(n)}
\big)_{L^2((X\times \mathbb R)^l,(m\otimes\nu)^{\otimes l})}\,l!\\
& = \sum_{\substack{\theta \in\Pi(n),\, |\theta| =l \\ \text{$\theta$ satisfying \eqref{hit98t}}}}
\big(\Sym_l(\mathcal E_\theta f^{(n)}),\mathcal E_\zeta f^{(n)}
\big)_{L^2((X\times \mathbb R)^l,(m\otimes\nu)^{\otimes l})}\,l!\\
& =\frac{1}{n!}\sum_{\substack{\theta \in\Pi(n),\, |\theta| =l \\ \text{$\theta$ satisfying \eqref{hit98t}}}}
\sum_{\pi \in S_n[\zeta, \theta]}\int_{X_\zeta^{(n)}} Q_\pi(x_1,\dots,x_n)f^{(n)}(x_{\pi^{-1}(1)}, \dots , x_{\pi^{-1}(n)})\\
& \quad \times f^{(n)}(x_1, \ldots, x_n)\,m_{\nu,\,\zeta}^{(n)}(dx_1\times\dots\times  dx_n)\\
& =\frac{1}{n!}\sum_{\pi \in S_n} \int_{X_\zeta^{(n)}} Q_\pi(x_1,\dots,x_n)f^{(n)}(x_{\pi^{-1}(1)}, \dots , x_{\pi^{-1}(n)})\\
&\quad\times
 f^{(n)}(x_1, \ldots, x_n)\,m_{\nu,\,\zeta}^{(n)}(dx_1\times\dots\times  dx_n)\\
&  = \int_{X_\zeta^{(n)}}(\Sym_n\, f^{(n)})f^{(n)}\,dm_{\nu,\,\zeta}^{(n)}\,.
\end{align*}
Thus,  Theorem~\ref{utu8} is proven.

\subsection{Proof of Theorem~\ref{fu7r7}}

Let us first prove the following 

\begin{lemma}\label{vfyt6}Let $h\in C_0(X)$ and $f^{n}\in B_0(X^n)$, $n\in\mathbb  N$. Then formulas \eqref{iut78}, \eqref{yur75} hold with
\begin{align*}
\mathbf J^+(h)\Sym_n f^{(n)}&=\Sym_{n+1}(h\otimes f^{(n)}),\\
\mathbf J^0(h)\Sym_n f^{(n)}&=\Sym_{n}\mathscr J^0(h)f^{(n)},\\
\mathbf J^-_1(h)\Sym_n f^{(n)}&=\Sym_{n-1}\mathscr J_1^-(h)f^{(n)},\\
\mathbf J^-_2(h)\Sym_n f^{(n)}&=\Sym_{n-1}\mathscr J_2^-(h)f^{(n)},
\end{align*}
Here \begin{multline}\label{tydr6ue}
(\mathscr J_2^-(h)f^{(n)})(x_1,\dots,x_{n-1})\\
:=\sum_{i=1}^n\int_X dy\,h(y) f^{(n)}(x_1,\dots,x_{i-1},y,x_i,\dots,x_{n-1})T_i(y,x_1,\dots,x_{n-1}),
\end{multline}
where for any $\theta\in\Pi(n-1)$
\begin{equation}\label{gftuoihcs}
T_i^{(n)}\restriction X\times  X_\theta^{(n-1)} :=\prod_{\theta_u\in\theta:\, \max\theta_u\le i-1}Q(y,x_{\theta_u}).
\end{equation}

\end{lemma}

\begin{proof}
  By \eqref{igyur7} and \eqref{f75re75},
we have
\begin{equation}
 \la \omega,h\ra=\mathcal A^+(h)+\mathcal A^0(h) +d\Gamma(M_h\otimes A^-)+a^-(h\otimes 1).\label{duvc}
\end{equation}

(i) ($\mathbf J^+(h)$ part)
From the proof of Proposition \ref{iyr867r}  it follows that
\begin{equation}\label{vfytde57e} 
\mathbf U^{-1}\mathcal A^+(h)\mathbf U\Sym_n f^{(n)}=\Sym_{n+1}(h\otimes f^{(n)})=\mathbf J^+(h)\Sym_n f^{(n)}.
\end{equation}

(ii) ({\it $\mathbf J^0(h)$ part})  By Lemma \ref{lgti8t}, Proposition \ref{iyr867r}, \eqref{iuagiuai},\eqref{f75re75}, \eqref{ir688u} and \eqref{ho9t97tg},
\begin{align}
\mathbf U^{-1}\mathcal A^0(h)\mathbf U\Sym_n f^{(n)}
& =\mathbf U^{-1}\mathcal A^0(h) \Sym \sum_{\theta\in\Pi(n)}\mathcal E_\theta f^{(n)}\notag\\
& =\mathbf U^{-1}\Sym N(M_h\otimes A^0)\sum_{\theta\in\Pi(n)}\mathcal E_\theta f^{(n)}\notag\\
&=  \mathbf U^{-1}\Sym \sum_{\theta\in\Pi(n)} \sum_{i=1}^n \mathcal E_\theta(h\times_i f^{(n)})b_{\gamma(i,\theta)-1}\, \gamma(i,\theta)^{-1}\notag\\
& =\Sym_n \mathscr J^0(h)f^{(n)}\notag\\
&=\mathbf J^0(h)\Sym_n f^{(n)}.\label{iwyeutf8}
\end{align}
Here,
$$ (h\times_i f^{(n)})(x_1,\dots,x_n):= h(x_i)f^{(n)}(x_1,\dots,x_n).$$

(iii) ({\it $\mathbf J^-_1(h)$ part})
Analogously,
\begin{align}
 &\mathbf U^{-1}d\Gamma(M_h\otimes A^-)\mathbf U\Sym_n f^{(n)}=
 \mathbf U^{-1}\Sym N(M_h\otimes A^-)\sum_{\theta\in\Pi(n)}\mathcal E_\theta f^{(n)}\notag\\
&\quad =\mathbf U^{-1} \Sym\sum_{l=1}^n\sum_{\substack{\theta\in\Pi(n)\\|\theta|=l}}\sum_{k=1}^l \mathbf 1^{\otimes (k-1)}\otimes(M_h\otimes A^-)\otimes \mathbf 1^{\otimes(l-k)}\mathcal E_\theta f^{(n)}\notag\\
&\quad=\mathbf U^{-1}\Sym\sum_{l=1}^{n-1}\sum_{\theta=\{\theta_1,\dots,\theta_l\}\in\Pi(n)}\sum_{\substack{k=1,\dots,l\\ |\theta_k|\ge2}}\mathbf 1^{\otimes (k-1)}\otimes(M_h\otimes A^-)\otimes \mathbf 1^{\otimes(l-k)}\mathcal E_\theta f^{(n)},\label{if7er5}
\end{align}
where \eqref{ft6e} is supposed to hold.
Note that, for $\theta=\{\theta_1,\dots,\theta_l\}\in\Pi(n)$ satisfying \eqref{ft6e} and $k\in\{1,\dots,l\}$ with $|\theta_k|\ge2$, we have
\begin{align}
&\big(\mathbf 1^{\otimes (k-1)}\otimes(M_h\otimes A^-)\otimes \mathbf 1^{\otimes(l-k)}\mathcal E_\theta f^{(n)}\big)(x_1,s_1,\dots,x_l,s_l)\notag\\
&\quad =a_{|\theta_k|-1}h(x_k)f^{(n)}_\theta(x_1,\dots,x_k,\dots,x_l)p_{|\theta_1|-1}(s_1)\dotsm p_{|\theta_{k-1}|-1}(s_{k-1})\notag\\
&\qquad\times p_{|\theta_k|-2}(s_k)
p_{|\theta_{k+1}|-1}(s_{k+1})
\dotsm p_{|\theta_{l}|-1}(s_{l}) .\label{iyur77rf}
\end{align}

Let us fix any $i,j\in\{1,\dots,n\}$ with $i<j$.  Consider the set
$$ L_i:=\{1,2,\dots,i-1,i+1,\dots,n\},$$
which has $n-1$ elements. Then any partition $\zeta=\{\zeta_1,\dots,\zeta_l\}\in\Pi(n-1)$ identifies a partition $\tilde\zeta=\{\tilde\zeta_1,\dots,\tilde\zeta_l\}$ of $L_i$: $\tilde\zeta_u:=K_i\zeta_u$, $u=1,\dots,l$, where
$$ K_i v:=\begin{cases} v,&\text{if }v\le i-1,\\
v+1,& \text{if }v\ge i.\end{cases}$$
 Let $\tilde \zeta_k$ be the element of $\tilde \zeta$ which contains $j$. Set
$$\theta_u:=\begin{cases}\tilde\zeta_u,&\text{if }u\ne k,\\
\tilde \zeta_k\cup\{i\},&\text{if }u=k.\end{cases}$$
Thus, we have constructed a partition $\theta=\{\theta_1,\dots,\theta_l\}\in\Pi(n)$ with $l\le n-1$.
Next, consider an arbitrary partition $\theta=\{\theta_1,\dots,\theta_l\}\in\Pi(n)$ with $l\le n-1$. Choose any $k\in\{1,\dots,l\}$ such that $|\theta_k|\ge2$. In how many ways can we obtain $\theta$ from $i,j$ and $\zeta\in\Pi(n-1)$ as above? This number is evidently equal to the number of all choices of
$i,j\in\{1,\dots,n\}$ with $i<j$ and $i,j\in \theta_k$, i.e.,
$$|\theta_k|(|\theta_k|-1)/2=(|\tilde\zeta_k|+1)|\tilde\zeta_k|/2=(|\zeta_k|+1)|\zeta_k|/2,
$$
where $j\in\tilde\zeta_k$, or equivalently $j-1\in\zeta_k$.
Hence, by \eqref{ft6e} \eqref{igftr7y}, \eqref{oduhuaijo}, \eqref{if7er5}, and \eqref{iyur77rf}, we get
\begin{equation}\label{jigtyl8tg}
\mathbf U^{-1}d\Gamma(M_h\otimes A^-)\mathbf U\Sym_n f^{(n)}
=\Sym_{n-1}\mathscr J_1^-(h)f^{(n)}=\mathbf J^-_1(h)\Sym_n f^{(n)}.
\end{equation}

(iv) ({\it $\mathbf J_2^-(h)$ part}).
For each $\theta=\{\theta_1,\dots,\theta_l\}\in\Pi(n)$ satisfying \eqref{ft6e}, we have
\begin{align}
&\big(a^-(h\otimes 1)\Sym_{l}\mathcal E_\theta f^{(n)}\big)(x_1,s_1,\dots,x_{l-1},s_{l-1})\notag\\
&\quad =\Sym_{l-1}\bigg(\int_Xdy\,
\sum_{\substack{i=1,\dots,l\\ |\theta_i|=1}}h(y)Q(y,x_1)Q(y,x_2)\dotsm Q(y,x_{i-1})\notag\\
&\qquad\times f^{(n)}_\theta(x_1,\dots,x_{i-1},y,x_i,\dots,x_{l-1})\notag\\
&\qquad \times p_{|\theta_1|-1}(s_1)\dotsm
p_{|\theta_{i-1}|-1|}(s_{i-1})p_{|\theta_{i+1}|-1}(s_i)\dotsm p_{|\theta_l|-1}(s_{l-1})\bigg),\label{agviyagv}
\end{align}
where we used \eqref{hyfd7urf} and \eqref{isgdiy}.
Hence, by \eqref{tydr6ue}, \eqref{gftuoihcs}, and \eqref{agviyagv},
\begin{equation}\label{kcgyiudgsc}
\mathbf U^{-1} a^-(h\otimes 1)\mathbf U \Sym_n  f^{(n)}=
\Sym_{n-1}\mathscr J_2^-(h)f^{(n)}=\mathbf J^-_2(h)\Sym_n  f^{(n)}.
\end{equation}
\end{proof}

\begin{lemma}\label{ugfcubc}
For any $h\in C_0(X)$ and $f^{(n)}\in \mathbf B_0^{Q}(X^n)$, we have
\begin{align} (\mathbf J_2^-(h)f^{(n)})(x_1,\dots,x_{n-1})&=(\mathscr  J_2^-(h)f^{(n)})(x_1,\dots,x_{n-1})\notag\\
&=n\int_X dy\, h(y)f^{(n)}(y,x_1,\dots,x_{n-1}).\label{trdtrs}\end{align}
\end{lemma}

\begin{proof} Fix any $n\ge 2$ and $i\in \{2,\dots,n\}$. Let a permutation $\pi\in \mathfrak S_n$ be given by $\pi(1)=i$,  $\pi(j)=j-1$ for  $j=2,\dots,i$, and $\pi(j)=j$ for $j=i+1,\dots,n$. Recall the operator $ \Psi_\pi$ defined in subsec.~\ref{yfd6rd6}. By  \eqref{uyr75rw} and \eqref{gftuoihcs}, we have, for each $(x_1,\dots,x_n)\in X^n$ such that $x_1\ne x_j$ for $j\in\{2,\dots,n\}$,
\begin{equation}\label{gufr8o}( \Psi_\pi f^{(n)})(x_1,\dots,x_n)=f^{(n)}(x_2,x_3,\dots,x_{i},x_1,x_{i+1},\dots,x_n)T_i(x_1,x_2,\dots,x_n).\end{equation}
Since $f\in \mathbf B_0^{Q}(X^n)$, by \eqref{gygfy} and
\eqref{gciugfi},
\begin{equation}\label{houygt9it}\Psi_\pi f^{(n)}=\Psi_\pi \Sym_n f^{(n)}=
\Sym_nf^{(n)}=f^{(n)}. \end{equation}
By \eqref{gufr8o} and \eqref{houygt9it}, for each $(x_1,\dots,x_{n-1})\in X^{n-1}$
\begin{align*}& \int_X dy\,h(y) f^{(n)}(x_1,\dots,x_{i-1},y,x_i,\dots,x_{n-1})T_i(y,x_1,\dots,x_{n-1})
\\&\quad =\int_{X\setminus\{x_1,\dots,x_{n-1}\}} dy\,h(y) f^{(n)}(x_1,\dots,x_{i-1},y,x_i,\dots,x_{n-1})T_i(y,x_1,\dots,x_{n-1})
\\
&\quad = \int_X dy\,h(y) f^{(n)}(y,x_1,\dots,x_{n-1}).\end{align*}
Hence, by \eqref{tydr6ue},
\begin{equation} (\mathscr  J_2^-(h)f^{(n)})(x_1,\dots,x_{n-1})=n\int_X dy\, h(y)f^{(n)}(y,x_1,\dots,x_{n-1})=:g^{(n-1)}(x_1,\dots,x_{n-1}).\label{fy7e5irfv}\end{equation}
Since $f^{(n)}\in \mathbf B_0^{Q}(X^n)$, formula \eqref{ft7ier5ird} holds for each $\pi\in\mathfrak S_n$. Hence, for each $\pi\in\mathfrak S_{n-1}$,
$$ g^{(n-1)}(x_1,\dots,x_{n-1})= Q_\pi(x_1,\dots,x_{n-1})g^{(n)}(x_{\pi^{-1}(1)},\dots,x_{\pi^{-1}(n)}),$$
see \eqref{uyr75rw}. Therefore,
\begin{equation}\label{guy8rt7o68r}
\Sym\, g^{(n-1)}=g^{(n-1)}.
\end{equation}
So, the lemma follows from \eqref{fy7e5irfv} and \eqref{guy8rt7o68r}.
\end{proof}

Now, Theorem~\ref{fu7r7} follows from Lemmas \ref{vfyt6}, \ref{ugfcubc}.

\subsection{Proof of Theorem \ref{urr8r}}

Assume that \eqref{vggyufd7u} holds.
Then, by \eqref{ho9t97tg} and \eqref{oduhuaijo}, we get
$R_i^{(n)}\equiv \lambda$ and  $S_{j-1}^{(n)}\equiv2\eta$. Hence, by \eqref{ir688u} and \eqref{igftr7y}, for any $h\in C_0(X)$, the operators $\mathscr J^0(h)$ and $\mathscr J_1^-(h)$ map $\mathcal F_{\mathrm{fin}}(C_0(X))$ into itself. Hence, condition (C) is satisfied. Furthermore, equality \eqref{yufu7edseaa} follows Theorem~\ref{fu7r7}.

To show that \eqref{vggyufd7u} is necessary for condition (C) to hold, we proceed as follows.
We first assume that the measure $\nu=\delta_\lambda$ for some $\lambda\in\mathbb R$ (Guassian/Poisson). Then $a_k=0$ for all $k\in\mathbb N$, $b_0=\lambda$, and the values of $b_k$ for $k\in\mathbb N$ maybe chosen arbitrarily. Thus, \eqref{vggyufd7u} holds in this case with $\eta=0$.

We next assume that the support of the measure $\nu$ contains an infinite number of points. Thus, $a_k>0$ for all $k\in\mathbb N$.

\begin{lemma}\label{hiwfgwy8} Let $q\ne-1$. Let $a_k>0$ for all $k\in\mathbb N$.
 Let $n\ge2$ and let $f^{(n)}\in C_0(X^n)$ be such that $\Sym_n\, f^{(n)}=0$ $m_\nu^{(n)}$-a.e.\ on the set $X^{(n)}_\theta$, where $\theta=\{\theta_1,\theta_2\}\in \Pi(n)$ with $\theta_1=\{1\}$ and $\theta_2=\{2,\dots,n\}$. Then $ f^{(n)}(x,\dots,x)=0$  for all $x\in X$.

  In the fermion case, $q=-1$, the above result remains true for $n\ge3$.
\end{lemma}

\begin{proof} Let $x_1,x_2\in X$ be such that $x_1^1<x_2^1$.  (Recall that $x^i$ denotes the $i$-th coordinate of $x=(x^1,\dots,x^d)\in X$.) In particular, $x_1<x_2$. Then
 \begin{align}
&(\Sym_n\,f^{(n)})(x_1,x_2,x_2,\dots,x_2)=\frac1n\big( f^{(n)}(x_1,x_2,x_2,\dots,x_2)+
f^{(n)}(x_2,x_1,x_2,\dots,x_2)\notag\\
&\quad +\dots+f^{(n)}(x_2,\dots,x_2,x_1,x_2)+q f^{(n)}(x_2,\dots,x_2,x_1)\big)=0.\label{tye65e}
\end{align}
Since the function $f^{(n)}$ is continuous, equality \eqref{tye65e} holds point-wise on the open set
$$\{(x_1,x_2)\in X^2\mid x_1^1<x_2^1\}.$$
Therefore, for all $x\in X$, we get
$\frac{n-1+q}n f^{(n)}(x,\dots,x)=0$.
Thus,  $f^{(n)}(x,\dots,x)=0$ if either $q\ne -1$ and $n\ge2$, or $q=-1$ and $n\ge3$.
\end{proof}

We now set $\lambda:=b_0$. Let us show  that, if (C) holds, then $b_k=\lambda(k+1)$ for all $k\in\mathbb Z_+$. The proof below works for any anyon statisics, however, in the case where $q\ne-1$, this proof can be significantly simplified.

Let $\varepsilon\in\mathbb R$ be such that $b_1=2\lambda+\varepsilon$. We will now show by induction that
\begin{equation}\label{igyfd7i5ei}
b_k=\lambda(k+1)+\varepsilon,\quad k\ge1.\end{equation}
 Assume that equality in \eqref{igyfd7i5ei} holds for $k=1,\dots,n$. Fix any  $h\in C_0(X)$ and $f^{(n+2)}\in C_0(X^{n+2})$.
 We define a function $g^{(n+2)}\in C_0(X^{n+2})$ by
\begin{equation}\label{tye66ie7}
g^{(n+2)}(x_1,\dots,x_{n+2}):=f^{(n+2)}(x_1,\dots,x_{n+2})\big(\lambda h(x_1)
+h(x_2)(\lambda(n+1)+\varepsilon)\big).\end{equation}
 Let $\theta=\{\theta_1,\theta_2\}\in\Pi(n+2)$ with $\theta_1=\{1\}$, $\theta_2=\{2,\dots,n+2\}$. By \eqref{ir688u} and   \eqref{ho9t97tg}, we have $m_\nu^{(n+2)}$-a.e.\ on $X_\theta^{(n+2)}$:
\begin{align*}&(\mathscr J^0(h)f^{(n+2)})(x_1,\dots,x_{n+2})\\
&\quad =f^{(n+2)}(x_1,\dots,x_{n+2})\big(\lambda h(x_1)
+(n+1)h(x_2)(\lambda(n+1)+\varepsilon)/(n+1)\big)\\
&\quad =g^{(n+2)}(x_1,\dots,x_{n+2}).
 \end{align*}
Since (C) holds, there exists a function $u^{(n+2)}\in C_{0}(X^{n+2})$ such that
\begin{equation}\label{iwgyGF}
\Sym_{n+2}\mathscr J^0(h)f^{(n+2)}=\Sym_{n+2}\,u^{(n+2)}\end{equation} $m_\nu^{(n+2)}$-a.e.\ on $X^{n+2}$. Hence,
$$ \SSym_{n+2}(g^{(n+2)}-u^{(n+2)})(x_1,\dots,x_{n+2})=0 $$
for $m_c^{(n+2)}$-a.a.\ $(x_1,\dots,x_{n+2})\in X^{(n+2)}_\theta$.
Noting that $g^{(n+2)}-u^{(n+2)}\in C_0(X^{n+2})$, we conclude from Lemma~\ref{hiwfgwy8} that
\begin{equation}\label{ifu7i5e} u^{(n+2)}(x,\dots,x)=g^{(n+2)}(x,\dots,x),\quad x\in X.\end{equation}
By \eqref{tye66ie7}--\eqref{ifu7i5e},
\begin{equation}\label{ftur75reeawa}(\mathscr J^0(h)f^{(n+2)})(x,\dots,x)=\big(\lambda (n+2)+\varepsilon)
h(x)f^{(n+2)}(x,\dots,x)\end{equation}
for all $x\in X$.
By \eqref{ir688u},  \eqref{ho9t97tg}, and \eqref{ftur75reeawa}, we therefore get $b_{n+1}=\lambda(n+2)+\varepsilon$.
Thus, \eqref{igyfd7i5ei} is proven.

Our next aim is to show that $\varepsilon=0$. We  first derive the following analog of Lemma~\ref{hiwfgwy8}.

\begin{lemma}\label{uit8ot}Let $a_k>0$ for all $k\in\mathbb N$.
 Let $f^{(5)}\in C_0(X^5)$ be such that $\Sym_5\, f^{(5)}=0$ $m_\nu^{(5)}$-a.e.\ on the set $X^{(5)}_\theta$, where $\theta=\{\theta_1,\theta_2\}\in\Pi(5)$ with $\theta_1=\{1,2\}$, $\theta_2=\{3,4,5\}$.
Then $f^{(5)}(x,\dots,x)=0$ for all $x\in X$.
\end{lemma}

\begin{proof}The proof is similar to that of Lemma~\ref{hiwfgwy8}. In fact, from the condition of Lemma~\ref{uit8ot}, we get
$ \frac{6+4q}{10}f^{(5)}(x,\dots,x)=0$,
which implies the statement.
\end{proof}

By \eqref{ir688u},    \eqref{ho9t97tg}, and \eqref{igyfd7i5ei}, we have, for $m_\nu^{(5)}$-a.e.\  $(x_1,\dots,x_5)\in X_\theta^{(5)}$ with $\theta\in\Pi(5)$ being as in Lemma~\ref{uit8ot},
\begin{equation}\label{vct6red6e} (\mathscr J^0(h)f^{(5)})(x_1,\dots,x_5)=
f^{(5)}(x_1,\dots,x_5)\big(h(x_1)(2\lambda+\varepsilon)+ h(x_3)(3\lambda+\varepsilon)\big).\end{equation}
Analogously to derivation of  formula \eqref{ftur75reeawa}, we conclude from condition (C),
Lemma~\ref{uit8ot}, and \eqref{vct6red6e} that, for all $x\in X$,
\begin{equation}\label{hur75r}
(\mathscr J^0(h)f^{(5)})(x,\dots,x)=f^{(5)}(x,\dots,x)
h(x)(5\lambda+2\varepsilon).
\end{equation}
On the other hand, by
\eqref{ir688u},    \eqref{ho9t97tg}, and \eqref{igyfd7i5ei},
we have, for all $x\in X$
\begin{equation}\label{it86tr}
(\mathscr J^0(h)f^{(5)})(x,\dots,x)=f^{(5)}(x,\dots,x)
h(x)(5\lambda+\varepsilon).
\end{equation}
Comparing \eqref{hur75r} and \eqref{it86tr}, we see that $\varepsilon$ must be equal to zero.

The proof of the equality $a_k=\eta k(k+1)$ for $k\in\mathbb N$ is similar, so we only outline it.
Denote $\eta:=a_1/2$.
Using Lemma~\ref{hiwfgwy8} and formulas \eqref{igftr7y}, \eqref{oduhuaijo}, we get the recursive formula
\begin{equation}\label{iyr8lotrg}
a_{n+1}=2\eta+\big((n+1)(n+2)-2\big)\frac{a_n}{n(n+1)}
\end{equation}
for $n\ge2$.
Choose $\varepsilon\in\mathbb R$ so that $a_2=6\eta+\varepsilon$. Then, by \eqref{iyr8lotrg},
\begin{equation}\label{oiwhcy97t}
a_3=12\eta+\frac{10}6\varepsilon,\quad
a_4=20\eta+\frac52\varepsilon,\quad a_5=30\eta+\frac72\varepsilon.
\end{equation}
On the other hand, by Lemma~\ref{uit8ot},
\begin{equation}\label{fur7}
a_5=a_2+2a_3.
\end{equation}
From  \eqref{oiwhcy97t} and \eqref{fur7}, we get $\varepsilon=0$. Hence, the recursive formula \eqref{iyr8lotrg} holds for all $n\ge1$. From here the desired equality follows.

We finally consider the case where
the support of the measure $\nu$ consists of $l$ points with $l\ge2$ being finite. In the case where $q=-1$, we will additionally assume that $l\ge3$.
  Then
 $a_1>0$, $a_2>0$,\dots,$a_{l-1}>0$, $a_i=0$ for$i\ge l$. Furthermore, by \eqref{yuft8uotfr8o},
 $ c_1>0$, $c_2>0$,\dots,$c_k>0$, $c_i=0$ for $i\ge l+1$.
 Let condition (C) be satisfied.
 Then, in view of the construction of the measures $m_\nu^{(n)}$, analogously to the above, we conclude that formula  \eqref{iyr8lotrg} holds for $n=1,2,\dots,l-1$. In particular, we get
 $$ a_l=a_1+\big( l(l+1)-2\big)\frac{a_{l-1}}{(l-1)l}\,.$$
Since $a_1>0$ and $a_{l-1}>0$, we therefore get $a_l>0$, which contradicts the fact that $a_l=0$. Thus, (C) can not be satisfied. Theorem \ref{urr8r} is proven.
 
 We leave the easy proof of Proposition~\ref{ytfr7} to the interested reader.  Let us show, however, how Theorem~\ref{ur7o67r6} can now be easily derived.

Assume $q=1$. Assume that $\mathscr{CP}=\mathscr{OCP}$. Then, for any
$h\in C_0(X)$ and $f^{(n)}\in C_0(X^n)$, we have
\begin{equation}\label{vytjdy6}\la \omega,h\ra\la P_n(\omega),f^{(n)}\ra\in \mathscr{OCP}\end{equation}
(we used that product of any polynomials from $\mathscr{CP}$ belongs to $\mathscr{CP}$). Since
$$\mathbf J_2^-(h)\la f^{(n)},P_n(\omega)\ra=\la \mathscr J^-_2(h)f^{(n)},P_{n-1}(\omega)\ra\in \mathscr{OCP},$$
we therefore conclude from Theorem \ref{fu7r7} and \eqref{vytjdy6} that (C) holds. Hence, by Theorem~\ref{urr8r}, \eqref{vggyufd7u} holds.

Let us now assume that \eqref{vggyufd7u} holds. Then, as follows from the proof of Theorem~\ref{urr8r}, $h\in C_0(X)$, the operators $\mathscr J^0(h)$ and $\mathscr J_1^-(h)$ map $\mathcal F_{\mathrm{fin}}(C_0(X))$ into itself. Hence, for any  $f^{(n)}\in C_0(X^n)$, \eqref{vytjdy6} holds. From here the equality $\mathscr{CP}=\mathscr{OCP}$ can be deduced analogously to the proof of \cite[Theorem~4.1]{BL1}.

\subsection{Proof of Theorem \ref{hfu8fr78}}
We will only prove equality \eqref{bkvgutfgi} as the proof of equality \eqref{vfyter} is similar and simpler. Note also that formula \eqref{ydfyd} will follow from \eqref{adiohdsrtdv}--\eqref{bkvgutfgi}.

It suffices to prove that, for any $h\in C_0(X)$,
$$ \mathbf J_1^-(h)g^{(n)}=\int_X dx\, h(x)\eta \partial_x^\dag \partial_x\partial_x\,g^{(n)},$$
where $g^{(n)}\in\mathbf B^{Q}_0(X^n)$ is of the form
$g^{(n)}=f_1\cd\dotsm \cd f_n$,
with $f_1,\dots,f_n\in B_0(X)$. We have
$$
g^{(n)}(x_1,\dots,x_n):
=\frac1{n!}\sum_{\pi\in\mathfrak S_n} Q_\pi (x_1,\dots,x_n)f_{\pi(1)}(x_1)\dotsm f_{\pi(n)}(x_n).
$$
Hence, by \eqref{igiugygg},
\begin{align}
& \left(\int_X dx\, h(x)\di_x^\dag\di_x\di_x g^{(n)}\right)(x_1,\dots,x_{n-1})\notag\\
& =\Sym_{n-1}\bigg(
\frac1{(n-2)!}\sum_{\pi\in\mathfrak S_n} Q_\pi(x_1,x_1,x_2,\dots,x_{n-1})\notag\\
&\qquad\qquad\qquad\times (hf_{\pi(1)}f_{\pi(2)})(x_1)
f_{\pi(3)}(x_2)\dotsm f_{\pi(n)}(x_{n-1})
\bigg)\notag\\
& =\sum_{1\le i<j\le n}\frac1{(n-2)!}\sum_{\substack{\pi\in\mathfrak S_n\\
\pi\{1,2\}=\{i,j\} }}
\Sym_{n-1}\big(
 Q_\pi(x_1,x_1,x_2,\dots,x_{n-1})\notag\\
&\qquad\qquad\qquad\qquad\qquad\qquad\qquad\qquad\times (hf_{i}f_{j})(x_1)
f_{\pi(3)}(x_2)\dotsm f_{\pi(n)}(x_{n-1})
\big).\label{bigi}
\end{align}
By \eqref{uyr75rw}, for any $\pi\in\mathfrak S_n$ satisfying $\pi\{1,2\}=\{i,j\}$ with $i<j$, and any $(x_1,x_2,\dots,x_{n-1})\in X^{n-1}$, we have
\begin{equation}\label{dyre6e}
 Q_\pi(x_1,x_1,x_2,\dots,x_{n-1})= Q_{\sigma_{ij}(\pi)}(x_1,x_2,\dots,x_{n-1}).
\end{equation}
Here the permutation $\sigma_{ij}(\pi)\in \mathfrak S_{n-1}$ is defined as follows:
$$\sigma_{ij}(\pi)(1):=j,$$
and for $k=2,\dots,n-1$,
$$\sigma_{ij}(\pi)(k):=\begin{cases}
\pi(k+1),&\text{if }\pi(k+1)<i,\\
\pi(k+1)-1,&\text{if }\pi(k+1)>i.
\end{cases}$$
By \eqref{dyre6e}, for any $\pi\in\mathfrak S_n$ satisfying $\pi\{1,2\}=\{i,j\}$ with $i< j$,
\begin{align}
&  Q_\pi(x_1,x_1,x_2,\dots,x_{n-1}) (hf_{i}f_{j})(x_1)
f_{\pi(3)}(x_2)\dotsm f_{\pi(n)}(x_{n-1})\notag\\
&\quad =  Q_{\sigma_{ij}(\pi)}(x_1,x_2,\dots,x_{n-1})
\big(f_1\otimes \dots\otimes f_{i-1}\otimes f_{i+1}\notag\\
&\qquad\otimes\dots\otimes f_{j-1}\otimes (hf_if_j)\otimes f_{j+1}\otimes\dots\otimes f_n \big)(x_{\sigma_{ij}(\pi)^{-1}(1)},\dots, x_{\sigma_{ij}(\pi)^{-1}(n-1)} )\notag\\
&\quad=  \Psi_{\sigma_{ij}(\pi)}\big(f_1\otimes \dots\otimes f_{i-1}\otimes f_{i+1}\notag\\
&\qquad \otimes\dots\otimes f_{j-1}\otimes (hf_if_j)\otimes f_{j+1}\otimes\dots\otimes f_n \big)(x_1,\dots,x_{n-1}).
\notag
\end{align}
Hence, by \eqref{gygfy} and \eqref{gciugfi},
\begin{align}
&\Sym\big( Q_\pi(x_1,x_1,x_2,\dots,x_{n-1}) (hf_{i}f_{j})(x_1)
f_{\pi(3)}(x_2)\dotsm f_{\pi(n)}(x_{n-1})\big)\notag\\
&\quad=(f_1\cd \dotsm \cd f_{i-1}\cd f_{i+1}\cd\dotsm\cd f_{j-1}\cd(hf_if_j)\cd f_{j+1}\cd\dotsm
\cd f_n)(x_1,\dots,x_{n-1}).\label{fcauyfy}
\end{align}
By \eqref{bigi} and \eqref{fcauyfy}, we thus get
\begin{align*}
&\int_X dx\, h(x)\di_x^\dag\di_x\di_x g^{(n)}\\
 &\quad =2\sum_{1\le i<j\le n} f_1\cd \dotsm \cd f_{i-1}\cd f_{i+1}\cd\dotsm\cd f_{j-1}\cd(hf_if_j)\cd f_{j+1}\cd\dotsm
\cd f_n.
\end{align*}
From here equality \eqref{bkvgutfgi} follows.

\begin{center}
{\bf Acknowledgements}\end{center}
M.B. and E.L. acknowledge the financial support of the Polish
National Science Center, grant no.\ Dec-2012/05/B/ST1/00626, and of
the SFB 701 ``Spectral
structures and topological methods in mathematics'', Bielefeld University. MB was partially supported
by the MAESTRO grant DEC-2011/02/A/ST1/00119.


\begin{thebibliography}{99}

\bibitem{afs}  Accardi, L., Franz, U., Skeide, M.:
 Renormalized squares of white noise and other non-Gaussian noises as L\'evy processes on real Lie algebras. Comm.\ Math.\ Phys.\ 228, 123--150 (2002)


 \bibitem{AKR} Albeverio, S., Kondratiev, Yu.G.,  R\"ockner, M.: Analysis and geometry on configuration spaces. J. Funct. Anal. 154,  444--500 (1998)


\bibitem{A1}  Anshelevich, M.:  $q$-L\'evy processes. J. Reine Angew. Math. 576, 181--207  (2004)

\bibitem{a2} Anshelevich, M.: Free Meixner states. Commun.\ Math.\ Phys.\ 276,  863--899 (2007)


\bibitem{a5} Anshelevich, M.: Orthogonal polynomials with a resolvent-type generating function. Trans.\ Amer.\ Math.\ Soc.\ 360,  4125--4143 (2008)

\bibitem{BBLS}  Belinschi, S.T., Bo\.zejko, M.,  Lehner, F.,  Speicher, R.:  The normal distribution is $\boxplus$-infinitely divisible. Adv. Math. 226,  3677--3698 (2011)

\bibitem{B}  Berezansky, Y.M.: Commutative Jacobi fields in Fock space.  Integral Equations Operator Theory 30, 163--190 (1998)


\bibitem{BK} Berezansky, Y.M., Kondratiev, Y.G.: Spectral methods in infinite-dimensional analysis. Vol. 1, 2.  Kluwer Academic Publishers, Dordrecht, 1995.

\bibitem{BML}  Berezansky, Y.M., Lytvynov, E., Mierzejewski, D.A.: The Jacobi field of a L\'evy process.  Ukrainian Math. J. 55, 853--858 (2003)

\bibitem{Biane} Biane, P.: Processes with free increments. 
Math. Z. 227,  143--174 (1998)

\bibitem{Bozejko}  Bo\.zejko, M.:  Deformed Fock spaces, Hecke operators and monotone Fock space of Muraki. Demonstratio Math. 45,  399--413 (2012)

\bibitem{bd} Bo\.zejko, M., Demni, N.: Generating functions of Cauchy-Stieltjes type for orthogonal polynomials. Infin. Dimens. Anal. Quantum Probab. Relat. Top. 12,  91--98 (2009)

\bibitem{bb} Bo\.zejko, M., Bryc, W. On a class of free L\'evy laws related to a regression problem. J. Funct.\ Anal.\ 236, 59--77 (2006)

\bibitem{BH} Bo\.zejko, M., Hasebe, T.:  On free infinite divisibility for classical Meixner distributions. Probab. Math. Statist. 33,  363--375 (2013)


\bibitem{BKS} Bo\.zejko, M.,  K\"ummerer, B., Speicher, R.: 
$q$-Gaussian processes: non-commutative and classical aspects. 
Comm. Math. Phys. 185,  129--154 (1997) 


\bibitem{BL1}  Bo\.zejko, M.,  Lytvynov, E.: Meixner class of non-commutative generalized stochastic processes with freely independent values. I. A characterization. Comm. Math. Phys. 292,  99--129 (2009)


\bibitem{BL2}  Bo\.zejko, M.,  Lytvynov, E.: Meixner class of non-commutative generalized stochastic processes with freely independent values. II. The generating function. Comm. Math. Phys. 302,  425--451 (2011)



\bibitem{BLW} Bo\.zejko, M.,  Lytvynov, E.,
Wysocza\'nski, J.:
Noncommutative L\'evy processes for generalized (particularly anyon) statistics,
Comm. Math. Phys. 313,  535--569 (2012)

\bibitem{BS}  Bo\.zejko, M., Speicher, R.: An example of a generalized Brownian motion. Comm. Math. Phys. 137,  519--531
  (1991)

\bibitem{Br1} Br\"uning, E.: When is a field a Jacobi field? A characterization of states on tensor algebras. Publ. Res. Inst. Math. Sci. 22, 209--246 (1986)

\bibitem{Br2}    Br\"uning, E.: On the construction of fields and the topological role of Jacobi fields. Rep. Math. Phys. 21,  143--158 (1985)

\bibitem{BW}  Bryc, W.,  Wesolowski, J.:  Conditional moments of $q$-Meixner processes. Probab. Theory Related Fields 131,  415--441  (2005)


\bibitem{Chihara} Chihara, T.S.: An introduction to orthogonal polynomials.
Mathematics and its Applications, Vol. 13. Gordon and Breach Science Publishers, New York--London--Paris, 1978

\bibitem{DL} Das S., Lytvynov, E.: Generalized stochastic processes with independent values via the projection spectral theorem, in preparation.

\bibitem{DOP} Di Nunno, G., {\O}ksendal, B., Proske, F.:  Malliavin calculus for L\'evy processes with applications to finance. Universitext. Springer-Verlag, Berlin, 2009.

\bibitem{VGG1}    Gel'fand, I.M., Graev, M.I., Vershik, A.M.: Models of representations of current groups.   Representations of Lie groups and Lie algebras (Budapest, 1971), 121--179, Akad. Kiad\'o, Budapest, 1985.

\bibitem{GV} Gel'fand, I.M., Vilenkin, N.Ya.: Generalized functions. Vol. 4: Applications of harmonic analysis.  Academic Press, New York,London, 1964.

\bibitem{GM} Goldin, G.A., Majid, S.: On the Fock space for nonrelativistic anyon fields and braided tensor products.
J. Math. Phys. 45, 3770--3787  (2004)



\bibitem{GS} Goldin, G.A., Sharp, D.H.: Diffeomorphism groups, anyon fields, and $q$ commutators. Phys. Rev. Lett. 76, 1183--1187  (1996)

\bibitem{Grigelionis} Grigelionis, B.: Processes of Meixner type. Lithuanian Math. J. 39,  33--41 (1999)

\bibitem{HKLV} Hagedorn, D., Kondratiev, Y., Lytvynov, E., Vershik, A.: 
Laplace operators in gamma analysis,  arXiv:1411.0162

\bibitem{HKPR} Hagedorn, D., Kondratiev, Y.,  Pasurek, T., R{\"o}ckner, M.: Gibbs states over the cone of discrete measures.  
J. Funct. Anal. 264,  2550--2583 (2013)

\bibitem{HKPS} Hida, T., Kuo, H.-H., Potthoff, J.,Streit, L.:  White noise. An infinite-dimensional calculus.  Kluwer, Dordrecht, 1993.

\bibitem{Ito} It\^o, K.: 
Spectral type of the shift transformation of differential processes with stationary increments. 
Trans. Amer. Math. Soc. 81, 253--263 (1956) 

\bibitem{IK} Ito, Y., Kubo, I.:
Calculus on Gaussian and Poisson white noises.
Nagoya Math. J. 111, 41--84 (1988)


\bibitem{Kal} Kallenberg, O.: Random measures. Akad.-Verl., Berlin, 1983.


 \bibitem{KSSU}  Kondratiev, Y.G.,  da Silva,J.L.,    Streit, L.,   Us, G.:  Analysis on Poisson and gamma spaces. Infin. Dimens. Anal. Quantum Probab. Relat. Top. 1, 91--117 (1998)


\bibitem{KL}   Kondratiev, Y.G.,  Lytvynov,  E.W.:  Operators of gamma white noise calculus. Infin. Dimens. Anal. Quantum Probab. Relat. Top.  3,  303--335 (2000)



 \bibitem{LM} Liguori, A., Mintchev, M.: Fock representations of quantum fields with generalized statistics. Comm. Math. Phys. 169,  635--652  (1995)

\bibitem{Ly1}  Lytvynov, E.: Multiple Wiener integrals and non-Gaussian white noises: a Jacobi field approach. Methods Funct. Anal. Topology 1, no. 1, 61--85 (1995)

\bibitem{Ly2}  Lytvynov, E.:  Polynomials of Meixner's type in infinite dimensions---Jacobi fields and orthogonality measures. J. Funct. Anal. 200,  118--149 (2003)

\bibitem{L2} Lytvynov, E.:  Orthogonal decompositions for L\'evy processes with an application to the gamma, Pascal, and Meixner processes. Infin. Dimens. Anal. Quantum Probab. Relat. Top. 6,  73--102 (2003)

\bibitem{LR} Lytvynov, E.,  Rodionova, I.: Lowering and raising operators for the free Meixner class of orthogonal polynomials. Infin. Dimens. Anal. Quantum Probab. Relat. Top. 12,  387--399 (2009) 

\bibitem{Meixner} Meixner, J.: Orthogonale Polynomsysteme mit einem besonderen
Gestalt der erzeugenden Funktion.  J. London Math.\ Soc.\  9,  6--13 (1934)



\bibitem{NS} Nualart, D, Schoutens, W.: Chaotic and predictable representations for L\'evy processes.  Stochastic Process. Appl.  90, 109--122  (2000).



\bibitem{RS2} Reed, M.,  Simon, B.: Methods of modern mathematical physics. II. Fourier analysis, self-adjointness. Academic Press, New York--London, 1975

\bibitem{Rodionova} Rodionova, I.: Analysis connected with generating functions of exponential type in one and infinite dimensions. Methods Funct. Anal. Topology 11,  275--297 (2005)


\bibitem{Schoutens} Schoutens, W.:
Stochastic processes and orthogonal polynomials.
Lecture Notes in Statistics, Vol. 146. Springer-Verlag, New York, 2000


\bibitem{Schoutens_Teugels} Schoutens, W., Teugels, J.L.:
L\'evy processes, polynomials and martingales.  
Comm. Statist. Stochastic Models 14,  335--349  (1998) 

\bibitem{Sko} Skorohod, A.V.:
Integration in Hilbert space. 
 Springer-Verlag, New York-Heidelberg, 1974.


\bibitem{Surgailis} Surgailis, D.: On multiple Poisson stochastic integrals and associated Markov semigroups. Probab. Math. Statist. 3,  217--239 (1984)


\bibitem{TsVY} Tsilevich, N, Vershik, A,  Yor, M.: An infinite-dimensional analogue of the Lebesgue measure and distinguished properties of the gamma process. J. Funct. Anal. 185, 274--296 (2001)


\bibitem{Vershik}  Vershik, A.M.: Does a Lebesgue measure in an infinite-dimensional space exist? Proc. Steklov Inst. Math. 259, 248--272 (2007)



\bibitem{VGG} Vershik, A.M., Gel'fand, I.M., Graev, M.I.: Representations of the group of diffeomorphisms. Russian Math. Surveys 30, 1--50 (1975)


\bibitem{VGG3} Vershik, A.M.,  Gel'fand, I.M., Graev, M.I.:  Representation of $SL(2, R)$, where $R$ is a
ring of functions. Uspehi Mat. Nauk 28, 83--128 (1973) [English translation in
``Representation Theory,'' London Math. Soc. Lecture Note Ser., Vol. 69, pp. 15--60,
Cambridge Univ. Press, Cambridge, UK, 1982.]


\bibitem{VGG2}   Vershik, A.M.; Gel'fand, I.M., Graev, M.I.: Commutative model of the representation of the group of flows $SL(2,\mathbf R)^X$ connected with a unipotent subgroup.  Funct. Anal. Appl. 17, 80--82 (1983)

\bibitem{VTs} Vershik, A.M., Tsilevich, N.V.: Fock factorizations and decompositions of the $L^2$ spaces over general L\'evy processes.  Russian Math. Surveys 58,  427--472 (2003)

 \end{thebibliography}
\end{document}